\documentclass[9pt]{article}
\usepackage{amsfonts, amsmath, amsthm, amstext, amssymb, amscd}
\usepackage[dvips]{color}
\usepackage{graphicx}
\usepackage[all]{xy}
\usepackage{color}
\usepackage{ifpdf}
\usepackage{mathrsfs}
\input epsf
\newtheoremstyle{thm}{0.32em}{0.32ex}{\it}{}{\bf}{.}{0.5em}{\thmname{#1}\thmnumber{ #2}\thmnote{ ({\rm #3})}}
\theoremstyle{thm}
\newtheorem{theorem}{Theorem}[section]

\newtheorem{lemma}[theorem]{Lemma}

\newtheorem{corollary}[theorem]{Corollary}

\newtheorem{proposition}[theorem]{Proposition}

\newcounter{chap}
\newcounter{sect}

\newtheorem{mainlemma}[theorem]{Main Lemma}

\usepackage{titlesec}
\usepackage[normalem]{ulem}
\titleformat{\section}
{\normalfont\large\bfseries}{\thesection}{1em}{}
\titlespacing{\section}{0cm}{0.6cm}{0.2cm}

\makeatletter
\def\enumerate{%
 \ifnum \@enumdepth >\thr@@\@toodeep\else
   \advance\@enumdepth\@ne
   \edef\@enumctr{enum\romannumeral\the\@enumdepth}%
     \expandafter
     \list
       \csname label\@enumctr\endcsname
       {\usecounter\@enumctr\def\makelabel##1{\hss\llap{##1}}%
         \addtolength{\parsep}{1.56pt}
         \addtolength{\listparindent}{0pt} 
         \addtolength{\itemsep}{-8pt} 
         \addtolength{\topsep}{-1pt}} 
 \fi}
\makeatother

\def\calO{\mathcal{O}}

\def\calS{\mathcal{S}}
\def\calT{\mathcal{T}}

\def\calL{\mathcal{L}}

\def\calU{\mathcal{U}}
\def\calV{\mathcal{V}}

\def\CC{\mathbb{C}}

\def\HH{\mathbb{H}}
\def\NN{\mathbb{N}}
\def\PP{\mathbb{P}}
\def\QQ{\mathbb{Q}}
\def\RR{\mathbb{R}}
\def\TT{\mathbb{T}}
\def\ZZ{\mathbb{Z}}
\def\EE{\mathbb{E}}

\let\ol\overline
\def\bPh{\vphantom{b}}

\renewcommand{\baselinestretch}{1.2}

\renewenvironment{proof}{\noindent{\em
Proof.}\hspace{0.3em}}{\hfill\qed\vspace{0ex}}

\newenvironment{proofof}[1]{\noindent{\em #1.}\hspace*{0.3em}}{\hfill\qed\vspace{0ex}}

\begin{document}

\renewcommand{\baselinestretch}{1.2}

\begin{center}
{\Large\bf A dynamical Mordell Lang property on the disk}\\
{\sc By Ming-Xi Wang}\\
{\it ETH Z\"urich,  University of Salzburg\\} \quad E-mail address:
supershankly@gmail.com
\end{center}

\begin{abstract}
We prove that two finite endomorphisms of the unit disk with degree
at least two have orbits with infinite intersections if and only if
they have a common iteration.
\end{abstract}

\begin{figure}[b]
\rule[-2.5truemm]{5cm}{0.1truemm}\\[1mm]
{\footnotesize {\it AMS Classification (2010):} Primary, 11Z05;
Secondary 37P05.
\par {\it Key words and phrases:} arithmetic dynamics, fundamental group, rational points, Blaschke product, Faltings' theorem, heights, monodromy, elliptic rational function.
\par The author was partially supported by the
 scholarship of ZGSM, a SNF grant and Austrian Science Fund(FWF): P24574. }
 \end{figure}


\bigskip
\noindent

\vspace{-3mm}

\section{Introduction}

In recent papers \cite{GTZ08} and \cite{GTZ10} Ghioca, Tucker and
Zieve proved the following theorem ''{\it two non-linear polynomials
have orbits with infinitely many intersections if and only if they
have a common iteration.}'' Moreover they have observed that this is
a dynamical analogue of the Mordell-Lang conjecture, and have
formulated a more general dynamical Mordell-Lang problem. In this
paper we prove a result that fits into this context.

Let $(\mathrm{End}(X), \circ)$ respectively $\mathrm{Aut}(X)$ be the
monoid of finite endomorphisms respectively the group of holomorphic
automorphisms of an analytic space $X$, and let $\calO_{f}(x)$ be
the set of orbits of $x \in X$ under $f \in \mathrm{End}(X)$. Finite
endomorphism of the unit disk are finite Blaschke products, namely
rational functions of the following form
\begin{eqnarray}\label{FBP}
f(z) = \varrho\prod_{i = 1}^{n} \frac{z - a_i}{1-\overline{a_i}\,z}
\end{eqnarray}
with $\varrho$ in the unit circle $\TT$, $n \hspace{-0.7mm} \in
\hspace{-0.7mm} \NN$ and $a_i \hspace{-0.7mm} \in
\hspace{-0.7mm}\mathbb{E},$ where $\EE$ is the unit disk. In
particular it follows that $\mathrm{End}(\EE) \hspace{-0.7mm}
\subset \hspace{-0.7mm} \mathrm{End}(\PP^1)$. We shall regard a
finite Blaschke product as an endomorphism of the unit disk, the
unit circle, the Riemann sphere or the mirror image of the unit disk
$\overline{\mathbb{E}}^c$, depending on corresponding contexts. We
shall prove

\begin{theorem}\label{disk}
Given $\{x, y\}\hspace{-0.7mm} \subset\hspace{-0.7mm} \PP^1$ and
$\{f, g\} \hspace{-0.7mm}\subset \hspace{-0.7mm}
\mathrm{End}(\EE)\mathopen\setminus\mathclose \mathrm{Aut}(\EE)$. If
$\calO_f(x)\hspace{-0.5mm} \cap\hspace{-0.5mm} \calO_g(y)$ is
infinite then $f$ and $g$ have a common iteration.
\end{theorem}

Together with the work of Ghioca-Tucker-Zieve\,(cf.\,\cite{GTZ08},
\cite{GTZ10}) we have

\begin{theorem}[Theorem \ref{disk} + \cite{GTZ10}]\label{simply connected}
Let $X$ be a simply connected open Riemann surface with the ideal
boundary $X^{\partial}$, $\{f,g\} \hspace{-0.7mm} \subset
\hspace{-0.7mm} \mathrm{End}(X)\mathopen\setminus\mathclose
\mathrm{Aut}(X)$ and $\{x, y\}\hspace{-0.7mm} \subset\hspace{-0.7mm}
X\hspace{-0.7mm} \cup\hspace{-0.7mm} X^{\partial}$. If the
intersection $\calO_{f}(x)\hspace{-0.5mm} \cap\hspace{-0.5mm}
\calO_{g}(y)$ is infinite then $f$ and $g$ have a common iteration.
\end{theorem}





The proof of Theorem \ref{disk} is based on two faces of the
endomorphism monoid $(\mathrm{End}(\EE), \circ)$. On the one hand
the factorization of any element of $(\mathrm{End}(\EE), \circ)$ is
very rigid, and on the other hand the assumption leads to special
factorizations of $f^i$ and of $g^j$ in $(\mathrm{End}(\EE), \circ)$
for all $\{i,j\}\hspace{-0.7mm} \subset \hspace{-0.7mm} \NN$. The
rigidity of factorization is given by the monodromy action of
fundamental groups, and the speciality of factorizations is a
consequence of the finiteness theorem of rational points.

In Section \ref{facts on diophantine geometry and analytic geometry}
we recall some preliminary results from Diophantine geometry and
analytic geometry. Section \ref{section descent} is devoted to the
proof of our main lemma. In Section \ref{elliptic function section}
we discuss elliptic rational functions, which is one major technical
difficulty of this piece of work.  We shall explain the rigidity in
Section \ref{sectionrigidity}, based on a joint work with Ng
\cite{NW08}. The speciality result will be proved in Section
\ref{sectionbilutichy}, based on Faltings' theorem, the Bilu-Tichy
criterion, Riemann's existence theorem, additivity of Euler
characteristic, the use of a real structure and a deformation
argument. In Section \ref{section heights} we prove a result on
heights that is used in the proof of the main theorem. Finally in
Section \ref{section proof} we present the proof of our main
theorem.

Throughout this paper $ \mathfrak{D}\!_f$ and $\mathfrak{d}_f$,
respectively, are the divisor of critical points and the set of
critical values of a finite map $f$. The support of a divisor $D$ is
denoted by $|D|$. The Riemann sphere, Gaussian plane, Poincar\'e
disk and the unit circle are denoted by $\PP^1$, $\CC$, $\EE$ and
$\TT$. The lattice $\ZZ \omega_1\mathopen +\mathclose \ZZ \omega_2$
is abbreviated by $\Lambda_{\omega_1,\, \omega_2}$, and
$\overline{X}$ may refer to the complex closure, algebraic closure
or the complex conjugation.
Write $\mathscr{C}_t, \mathscr{B}r_a$ and $\mathscr{B}r_t$ for the
categories of continuous mappings of topological spaces, finite maps
of analytic spaces and branched coverings of topological spaces,
accordingly. Given $S$ an object of a category $\mathscr{C}$ we set
$\mathscr{C}/S$ to be the category of $\mathscr{C}$-morphisms over
$S$. Chebyshev polynomial of degree $n$ is denoted by $T_n$. For any
$a\hspace{-0.7mm} \in\hspace{-0.7mm} \EE$ we let
$\iota_{a}(z)\hspace{-0.7mm}=\hspace{-0.7mm}(z\mathopen+\mathclose
a)/(1\mathopen +\mathclose {} \overline{a}z).$ The 2-torsion points
of an elliptic curve $E$ are denoted by $E[2]$. A curve of type $(g,
\nu)$ is of genus $g$ and of $\nu$ points at infinity.


\section{Facts from Diophantine and analytic geometry}\label{facts on diophantine geometry and analytic
geometry}
Integral points of a complex irreducible projective curve $X$ are
{\it potentially dense} if there exists a field $k$ of finite type
over $\QQ$ such that $X(k)$ is infinite, while integral points of a
complex affine curve $X$ of type $(g, \nu)$ are {\it potentially
dense} if there exists $R$ of finite type over $\ZZ$ and an affine
curve $Y$ over $R$ such that $Y(\CC)$ is birational to $X$ and
$Y(R)$ is infinite. We collect celebrated theorems of Siegel and of
Faltings in

\begin{theorem}[Siegel \cite{S29}, Faltings \cite{Fa83}]\label{Siegel Faltings}
Integral points of an algebraic curve $X$ of type $(g, \nu)$ are
potentially dense if and only if $\chi(X)\hspace{-0.7mm}
=\hspace{-0.7mm} 2\hspace{-0.7mm} -\hspace{-0.7mm} 2g\hspace{-0.7mm}
-\hspace{-0.7mm} \nu\hspace{-0.7mm} \ge\hspace{-0.7mm} 0$.
\end{theorem}

There are only four types of curves with non-negative Euler
characteristic, namely ones of
\begin{eqnarray*}
(0, 0), (0, 1), (0, 2) \ \ \text{and} \ \ (1, 0).
\end{eqnarray*}
We shall call a curve of type $(0, 1)$ or $(0, 2)$ respectively of
$(0,0)$ or $(1,0)$ a Siegel factor respectively a Faltings factor.

A holomorphic map between Riemann surfaces is finite if and only if
it is non-constant and proper. A holomorphic map $f\mathpunct{:}
M\hspace{-0.7mm} \to\hspace{-0.7mm} N$ between Riemann surfaces is
finite if and only if there exists an integer $n$ such that
$f(z)\hspace{-0.7mm}=\hspace{-0.7mm}c$ has $n$ solutions for any
point $c$ of $N$. We shall define the number $n$ given above to be
the {\it degree} of $f$ and denote it by $\deg{f}$. We point out
that there are no finite maps between $\CC$ and $\mathbb{E}$, which
is a consequence of Liuville's theorem and

\begin{lemma}\label{disk and complex plane}
If there exists a finite map $f\mathpunct{:}
\mathbb{E}\hspace{-0.7mm} \to\hspace{-0.7mm} N$ then $N$ is
biholomorphic to $\mathbb{E}$.
\end{lemma}

For a proof we refer to \cite{NW08}. We shall need

\begin{lemma}[\cite{S}]\label{correspondence, group}\label{S}
Let $\mathfrak{d}$ be a discrete subset of $N$ and let
$q\hspace{-0.7mm} \in\hspace{-0.7mm} N\mathopen \setminus \mathclose
\mathfrak{d}$. There is a one-to-one correspondence between finite
maps $f\mathpunct{:} (M, p)\hspace{-0.7mm} \to \hspace{-0.7mm} (N,
q)$ of degree $n$ with $\mathfrak{d}_{f}\hspace{-0.7mm} \subset
\hspace{-0.7mm}\mathfrak{d}$ and subgroups $H$ of $\pi_{1}
(N\mathopen\setminus\mathclose\mathfrak{d}, q)$ of index $n$ given
by $f \mapsto H\hspace{-0.6mm} =\hspace{-0.6mm}
\pi_1(M\mathopen\setminus\mathclose f\!^{-\!1}(\mathfrak{d}), p)$.
\end{lemma}

A finite map $f\mathpunct{:} M\hspace{-0.7mm} \to\hspace{-0.7mm} N$
is called {\it linear} if $\deg{f}\hspace{-0.7mm} =\hspace{-0.7mm}
1,$ and a nonlinear finite map $f$ is called {\it
factorized}\,(resp.\,{\it prime} or {\it irreducible}) if there
exist\,(resp.\,exist no)  nonlinear finite maps $g \mathpunct{:} T
\hspace{-0.7mm} \rightarrow\hspace{-0.7mm} N$ and $h \mathpunct{:}
M\hspace{-0.7mm} \to\hspace{-0.7mm} T$ for which $f\mathopen
=\mathclose g\mathopen \circ\mathclose h$. The factorability of a
polynomial, as observed by Ritt in \cite{R22}, is determined by the
action of fundamental groups. By Lemma \ref{S} we slightly
generalize this fact to finite maps

\begin{theorem}[Ritt \cite{R22}, Ng-Wang \cite{NW08}]\label{Rittfirst}
Let $f\mathpunct{:} M\hspace{-0.7mm} \to\hspace{-0.7mm} N$ be a
finite map, $q\hspace{-0.7mm} \in\hspace{-0.7mm} N\mathopen
\setminus\mathclose \mathfrak{d}_f$ and $p\hspace{-0.7mm}
\in\hspace{-0.7mm} f^{-\!1}(q)$. The map $f$ is factorized if and
only if there exists a proper intermediate group between
$\pi_1(M\mathopen\setminus\mathclose f\!^{-\!1}(\mathfrak{d}_f), p)$
and  $\pi_{1} (N\mathopen\setminus\mathclose \mathfrak{d}_f , q)$.
\end{theorem}

This simple fact suggests the rigidity of the decomposition of
finite maps. Let $f\mathpunct{:} M\hspace{-0.7mm} \to\hspace{-0.7mm}
N$ be a finite map of degree $n$ and $q\hspace{-0.7mm}
\in\hspace{-0.7mm} N\mathopen \setminus\mathclose \mathfrak{d}_f$.
The natural group homomorphism $\rho \mathpunct{:} \pi_{1} (
\mathfrak{N}\mathopen\setminus\mathclose \mathfrak{d}_f ,
q)\hspace{-0.6mm} \to\hspace{-0.6mm} S_{n}$ which is called the {\it
monodromy}, and the image of $\rho$ is called the {\it monodromy
group} of $f$. With an additional assumption there is an even
stronger rigid property than the one stated in Theorem
\ref{Rittfirst}. Writing $\mathfrak{L}_n\hspace{-0.7mm}
=\hspace{-0.7mm} \{t\hspace{-0.7mm} \in\hspace{-0.7mm} \NN :
t\,|\,n\}$ for the lattice that $i\hspace{-0.7mm} \le\hspace{-0.7mm}
j $ if and only if $i\, | j$, we have

\begin{theorem}[Ritt \cite{R22}, Ng-Wang \cite{NW08}]\label{Rittsecond}
Let $f\mathpunct{:} M\hspace{-0.7mm} \to\hspace{-0.7mm} N$ be a
finite map and let $q\hspace{-0.7mm} \in\hspace{-0.7mm} N\mathopen
\setminus\mathclose \mathfrak{d}_f$. If there exists $\alpha
\hspace{-0.7mm} \in \hspace{-0.7mm} \pi_{1} (
N\mathopen\setminus\mathclose \mathfrak{d}_f , q)$ such that the
monodromy action of $\alpha$ is transitive then the lattice of
intermediate groups between $\pi_1(M\mathopen\setminus\mathclose
f\!^{-\!1}(\mathfrak{d}_f), p)$ and  $\pi_{1} (
N\mathopen\setminus\mathclose \mathfrak{d}_f , q)$ is a sublattice
of $\mathfrak{L}_{\deg{f}}$.
\end{theorem}

Finite map can be recovered from their monodromy by the ``Schere und
Kleister'' surgery \cite[p.41]{S}, and this is the well-known

\begin{theorem}[Riemann's existence theorem]\label{monodromy}
Let $N$ be a Riemann surface, $\mathfrak{d}\hspace{-0.7mm}
\subset\hspace{-0.7mm} N$ a discrete subset,
$q\hspace{-0.7mm}\in\hspace{-0.7mm}
N\mathopen\setminus\mathclose\mathfrak{d}$ and $\rho \mathpunct{:}
\pi_1(N\mathopen\setminus\mathclose\mathfrak{d}, q)\hspace{-0.7mm}
\to\hspace{-0.7mm} S_n$ a transitive representation. There exists a
unique pointed finite map $f\mathpunct{:} (M, p)\hspace{-0.7mm}
\to\hspace{-0.7mm} (N,q)$ between Riemann surfaces with the
monodromy of $f$ given by $\rho$.
\end{theorem}

We call the following group homomorphism $\rho_{T}\mathpunct{:}
F_2\hspace{-0.7mm} =\hspace{-0.7mm} \left<\sigma, \tau
\right>\hspace{-0.7mm} \to\hspace{-0.7mm} S_n$ a {\it Chebyshev
representation}: if $n\hspace{-0.7mm} =\hspace{-0.7mm} 2k$ then
\vspace{-0.2cm}
\begin{eqnarray*}
\rho_{T}(\sigma)\hspace{-2.5mm} &=&\hspace{-2.5mm} (2,2k)(3,2k\mathopen-\mathclose 1)\cdots(k,k\mathopen+\mathclose 2) \\
\rho_{T}(\tau)\hspace{-2.5mm} &=&\hspace{-2.5mm}
(2,1)\,(3,2k)\cdots(k\mathopen+\mathclose 1,k\mathopen+\mathclose 2)
\end{eqnarray*}
and if $n\hspace{-0.7mm} =\hspace{-0.7mm} 2k\mathopen +\mathclose 1$
then \vspace{-0.3cm}
\begin{eqnarray*}
\rho_{T}(\sigma)\hspace{-2.5mm} &=&\hspace{-2.5mm} (2,2k\mathopen +\mathclose 1)(3,2k)\cdots(k\mathopen +\mathclose 1,k\mathopen +\mathclose 2)\\
\rho_{T}(\tau)\hspace{-2.5mm} &=&\hspace{-2.5mm} (2,1)(3,2k\mathopen
+\mathclose 1)\cdots(k\mathopen +\mathclose 1,k\mathopen +\mathclose
3).
\end{eqnarray*}

If $X$ is a simply connected Riemann surface then $\pi_1(X\mathopen
\setminus\mathclose \{\text{2pts}\})$ is a free group of rank 2.
Theorem \ref{monodromy} played with $\rho_{T}\mathpunct{:}
\CC\mathopen \setminus\mathclose \{\text{2pts}\}\hspace{-0.7mm}
\to\hspace{-0.7mm} S_n$\,(resp.\,$\rho_{T}\mathpunct{:} \EE\mathopen
\setminus\mathclose \{\text{2pts}\}\hspace{-0.7mm}
\to\hspace{-0.7mm} S_n$) gives elements in
$\mathrm{End}(\CC)$\,(resp.\,$\mathrm{End}(\EE)$). The former are
polynomials associated to $T_n$, and the latter are called {\it
Chebyshev-Blaschke products}. This construction appeared in
\cite{W08} and \cite{NW08}. Let $k$ be the classical elliptic
modulus function as defined in \cite[p.99]{Chan}, then we set $
\gamma(t)\hspace{-0.7mm} =\hspace{-0.7mm}
k^{\frac{1}{2}}\left(4ti/\pi\right)$. In \cite{NW08}(or \cite{W08})
we have proved that

\begin{proposition}[Ng-Wang]\label{nomalized Chebyshev-Blaschke products,uniqueness}
Given $t\hspace{-0.7mm} >\hspace{-0.7mm} 0, n\hspace{-0.7mm}
\in\hspace{-0.7mm} \NN$ there is a unique $\mathcal{T}_{n,\,
t}\hspace{-0.7mm} \in\hspace{-0.7mm} \mathrm{End}(\EE)$ that is
characterized by properties that $\mathcal{T}^{-\!1}_{n,\,
t}[-\!\gamma(nt) , \gamma(nt)]\hspace{-0.7mm} =\hspace{-0.7mm}
[-\gamma(t), \gamma(t)]$ and that $\mathcal{T}_{n,\,
t}(\gamma(t))\hspace{-0.7mm} =\hspace{-0.7mm} \gamma(nt)$. These
$\mathcal{T}_{n,\, t}$ are Chebyshev-Blaschke products. If $f$ is a
Chebyshev-Blaschke product of degree $n$, then there exist
$\{\epsilon, \varepsilon\} \hspace{-0.7mm} \subset \hspace{-0.7mm}
\mathrm{Aut}(\EE)$ and $t\hspace{-0.7mm}>\hspace{-0.7mm}0$ such that
$f\hspace{-0.7mm} =\hspace{-0.7mm} \epsilon\hspace{-0.7mm}
\circ\hspace{-0.7mm} \mathcal{T}_{n,\,t}\hspace{-0.7mm}
\circ\hspace{-0.7mm} \varepsilon$.
\end{proposition}

These  $\mathcal{T}_{n,\, t}$ are called {\it normalized
Chebyshev-Blaschke products.}

\section{The main lemma}\label{section descent}


We shall make use of the following version of Riemann's covering
principle as given in \cite[p.119-120]{AS60}. Here a Riemann surface
is a pair $(X, \phi)$ with $X$ a connected Hausdorff space and
$\phi$ a complex structure, see \cite[p.144]{AS60}. However we shall
simply write $\mathbb{E}$ and $\CC$ when $\phi$ is the canonical
one.

\begin{theorem}[Riemann's covering principle]\label{Riemann's covering
principle} If $f \mathpunct{:} X_1\hspace{-0.7mm} \to\hspace{-0.7mm}
X_2$ is a covering surface and if $\phi_2$ is a complex structure on
$X_2$. Then there exists a unique complex structure $\phi_1$ on
$X_1$ such that $f\mathpunct{:}(X_1, \phi_1)\hspace{-0.7mm}
\to\hspace{-0.7mm} (X_2, \phi_2)$ is holomorphic.
\end{theorem}
Let $i_0\hspace{-0.7mm} \in \hspace{-0.7mm}
\mathrm{Hom}_{\mathscr{C}_t}(\EE, \CC)$ and $f\hspace{-0.7mm}
\in\hspace{-0.7mm} \mathrm{End}(\EE)$. Theorem \ref{Riemann's
covering principle} applied to $i_0 \hspace{-0.7mm} \circ
\hspace{-0.7mm} f\mathpunct{:} \EE \hspace{-0.7mm} \to
\hspace{-0.7mm} \CC$ gives a new complex structure $\phi$ on
$\mathbb{E}$ and a finite map $(\EE, \phi)\hspace{-0.7mm} \to
\hspace{-0.7mm} \CC$. The classical uniformization theorem together
with Lemma \ref{disk and complex plane} shows that $(\mathbb{E},
\phi)$
 must be the complex plane.  Writing $i_1\mathpunct{:}\EE\hspace{-0.7mm}\to\hspace{-0.7mm}(\mathbb{E}, \phi)\hspace{-0.7mm}=\hspace{-0.7mm}\CC$ for the topological
identity map, there exists a holomorphic map $(i_1, i_0)_{*}f$ which
makes the following diagram \bgroup \arraycolsep=1pt
\begin{displaymath}
\xymatrix{
 \mathbb{E} \ar[d]^{i_1}  \ar[r]^{f} & \EE \ar[d]^{i_0} \\
 \CC   \ar[r]^{(i\!_1\!,
i\!_0)\!_{*}\!f}      & \CC}
\end{displaymath}
\egroup \noindent commutative. We shall call $i_1$ a $f$-{\it
lifting} of $i_0$ and $(i_1, i_0)_{*}f$ a $(i_1, i_0)$-{\it descent}
of $f$.

The uniqueness in Theorem \ref{Riemann's covering principle} implies
that if $i_1, i_1'$ are two $f$-liftings of $i_0$ then there exists
a holomorphic isomorphism $\sigma\mathpunct{:}\CC\hspace{-0.7mm} \to
\hspace{-0.7mm}\CC $ such that $\sigma\mathopen \circ\mathclose
i_1\hspace{-0.7mm} = \hspace{-0.7mm} i_1'$. This gives
\begin{corollary}\label{deformation, uniqueness}
Let $i_0\hspace{-0.7mm} \in \hspace{-0.7mm}
\mathrm{Hom}_{\mathscr{C}_t}(\EE, \CC)$, $f\hspace{-0.7mm}
\in\hspace{-0.7mm} \mathrm{End}(\EE)$ and $i_1, i_1'$ both
$f$-liftings of $i_0$. There exists $\sigma\hspace{-0.7mm}
\in\hspace{-0.7mm} \mathrm{Aut}(\CC)$ which makes the following
diagram
\begin{displaymath}
\xymatrix{
  & \mathbb{E} \ar[dl]_{i_1'} \ar[d]^{i_1} \ar[r]^{f}  & \mathbb{E} \ar[d]^{i_0} \\
  \CC\ar@/_{0.5pc}/[rr]_{(i\!_1\!{'}\!, i\!_0)\!_*\!f}   \ar[r]^{\sigma}      & \CC \ar[r]^{(i\!_1\!,i\!_0)\!_*\!f}  & \CC}
\end{displaymath}
commutative.
\end{corollary}

\noindent Note that $(i_1, i_0)_*f$ and $(i_1', i_0)_*f$ are finite
self maps of $\CC$ and therefore are given by polynomials. The above
discussions remain true if we interchange $\EE$ with $\CC$, and then
one may check easily the following simple fact

\begin{proposition}\label{deformation of one factorization, converse}
Let $i_0\hspace{-0.7mm} \in \hspace{-0.7mm}
\mathrm{Hom}_{\mathscr{C}_t}(\CC, \EE)$ and $\{f_1,
f_2\}\hspace{-0.7mm} \subset\hspace{-0.7mm} \mathrm{End}(\CC)$ that
$f\hspace{-0.7mm} =\hspace{-0.7mm} f_1\mathopen\circ\mathclose f_2$.
If $i_1$ respectively $i_2$ is a $f_1$-lifting of $i_0$ respectively
a $f_2$-lifting of $i_1$ then $i_2$ is a $f$-lifting of $i_0$ and
$(i_2, i_0)_*f\hspace{-0.7mm} =\hspace{-0.7mm} (i_1,
i_0)_*f_{1}\mathopen\circ\mathclose(i_2, i_1)_*f_2$, as a relation
in $(\mathrm{End}(\EE), \circ)$.
\end{proposition}

Given $\{f, g\}\hspace{-0.7mm} \subset\hspace{-0.7mm}
\mathrm{End}(\EE)$ the curve $\PP^1\mathopen \times_{f,\,
g}\mathclose \PP^1$ is a double of $\mathbb{E}\mathopen \times_{f,\,
g}\mathclose \mathbb{E}$. Indeed, setting  $X^{\vee}$, $X$,
$X^{\partial}$ and $X^{\iota}$ for $\PP^1\mathopen \times_{f,\,
g}\mathclose \PP^1$, $\mathbb{E}\mathopen \times_{f,\, g}\mathclose
\mathbb{E}$, $\mathbb{T}\mathopen \times_{f,\, g}\mathclose
\mathbb{T}$ and $\overline{\mathbb{E}}^c\mathopen \times_{f,\,
g}\mathclose{} \overline{\mathbb{E}}^c$ we shall have
\begin{eqnarray*}
X^{\vee}\hspace{-0.7mm} =\hspace{-0.7mm} X\hspace{-0.7mm}
\cup\hspace{-0.7mm} X^{\partial}\hspace{-0.7mm} \cup\hspace{-0.7mm}
X^{\iota}.
\end{eqnarray*}
Take $i\hspace{-0.7mm} \in \hspace{-0.7mm}
\mathrm{Hom}_{\mathscr{C}_t}(\EE, \CC)$ and let $j_1\hspace{-0.7mm}
\in \hspace{-0.7mm} \mathrm{Hom}_{\mathscr{C}_t}(\EE,
\CC)$\;(resp.\;$j_2\hspace{-0.7mm} \in \hspace{-0.7mm}
\mathrm{Hom}_{\mathscr{C}_t}(\EE, \CC)$) be a
$f$-lifting\;(resp.\;$g$-lifting) of $i$. Setting
$X_{*}\hspace{-0.5mm} =\hspace{-0.5mm} \CC\mathopen \times_{(j_1,\,
i)_{*}f, \, (j_2,\, i)_{*}g}\mathclose \CC$ we will compare
algebraic components of the projective curve $X^{\vee}$ with those
of the affine curve $X_{*}$. It would be helpful to have in mind
that $X^{\vee}$, $X$ and $X_{*}$ are fibrations over $\PP^1$, $\EE$
and $\CC$, accordingly. This implies that $X^{\vee}$ is a double of
$X$ and $X_{*}$ equals $X$ in topology. Our main lemma gives an
arithmetic reflection of these simple facts.

\begin{mainlemma}\label{Siegel-Faltings}
There is a one-one correspondence between Faltings factors of
$X^{\vee}$ and Siegel factors of $X_{*}$.
\end{mainlemma}


\begin{proof}
We shall establish bijections from analytic components of $X$
firstly to algebraic components of $X^{\vee}$, and secondly to
algebraic components of $X_{*}$.

If $Y$ is an analytic component of $X$ then
$Y^{\iota}\hspace{-0.7mm} :=\hspace{-0.7mm} \{ (x, y) |
(1/\overline{x}, 1/\overline{y})\hspace{-0.5mm} \in\hspace{-0.5mm} Y
\}$ is an analytic component of $X^{\iota}$, as $ (x
,y)\hspace{-0.7mm} \in\hspace{-0.7mm} X \Leftrightarrow
(1/\overline{x}, 1/\overline{y})\hspace{-0.7mm} \in\hspace{-0.7mm}
X^{\iota}. $ The algebraic irreducible component $Y^{\vee}$ of
$X^{\vee}$ which contains $Y$ is given by $\overline{Y
\hspace{-0.7mm} \cup\hspace{-0.7mm} Y^{\iota}}$ and the
correspondence given by $Y\hspace{-0.7mm} \mapsto\hspace{-0.7mm}
Y^{\vee}$ is the first bijection as wanted.

Now set $Y_{*}\hspace{-0.7mm} =\hspace{-0.7mm} \{ (j_1(x), j_2(y)) |
(x, y)\hspace{-0.7mm} \in\hspace{-0.7mm} Y \}$ which is a subset of
$X_{*}$. The analytic structure involved is topological in nature,
and therefore $Y_{*}$ is also an analytic (and algebraic) component
of $X_{*}$. Here $Y\hspace{-0.7mm} \mapsto\hspace{-0.7mm} Y_{*}$
gives our second bijection.

In the first bijection $Y^{\vee}$ is a double of $Y$ which leads to
$\chi(Y^{\vee})\hspace{-0.7mm} =\hspace{-0.7mm} 2\chi(Y)$. In the
second one $Y_*$ is topologically equivalent to $Y$, and this gives
$\chi(Y_{*})\hspace{-0.7mm} =\hspace{-0.7mm} \chi(Y)$. Finally we
have $\chi(Y^{\vee})\hspace{-0.7mm} =\hspace{-0.7mm} 2\chi(Y_{*})$,
which together with Theorem \ref{Siegel Faltings} of Siegel and of
Faltings proves our assertion.
\end{proof}

\section{Facts on elliptic rational functions}\label{elliptic function section}
To handle normalized Chebyshev-Blaschke products $\mathcal{T}_{n,\,
t}\hspace{-0.7mm} \subset\hspace{-0.7mm} \mathrm{End}(\EE)$ recalled
in Section \ref{facts on diophantine geometry and analytic
geometry}, we shall treat them as descents of isogenies of elliptic
curves.

The construction of Chebyshev-Blaschke products\,(cf.\,\cite{NW08})
relies on the representation of fundamental groups. Indeed Zolotarev
constructed\,(cf.\,\cite{Zol77}) much earlier another family of
functions by using Jacobian elliptic functions, which was called
{\it Zolotarev fractions} by Bogatyrev\,(cf.\,\cite{Bo10}) or {\it
elliptic rational functions} by scientists working in filter
designs\,(cf.\,\cite{LTE01}). In \cite{NW08} we slightly generalized
Zolotarev's original construction and obtained a larger family of
rational functions $\mathcal{T}_{n, \tau}(n\hspace{-0.7mm}
\in\hspace{-0.7mm} \NN, \tau\hspace{-0.7mm} \in \hspace{-0.7mm}
\HH)$ by descents of cyclic isogenies of elliptic curves, where
Zolotarev's fractions correspond to $\mathcal{T}_{n, \tau}$ that
with $\tau$ purely imaginary. We verified in \cite{NW08} that there
is a canonical bijection between $\mathcal{T}_{n,\, t} (
t\hspace{-0.7mm}
>\hspace{-0.7mm} 0)$ and  $\mathcal{T}_{n,
\tau}(\tau \ \text{purely imaginary})$. Two entirely different
constructions, via elliptic functions\,(resp.\,fundamental groups)
taken by Zolotarev\,(resp.\,Ng-Wang), finally lead to essentially
the same class of functions.

The use of descents of cyclic isogenies of elliptic curves is
originally due to Zolotarev, but he only considered Jacobian
elliptic integrals (or functions) with real modulus $k$ which
prevent him from constructing a lager and universal family. For
classical special functions such as $\omega_1, \omega_2, e_i,
\mathrm{cn}, \mathrm{dn}$ we refer to \cite[Chapter V\!I\!I]{Chan},
and for more details of the following construction we refer to
\cite{NW08}. For $\tau\hspace{-0.7mm} \in\hspace{-0.7mm} \HH$ we
denote by $E_{\tau}$ respectively $E_{\tau}'$ for elliptic curve
$\CC/\Lambda_{1, \tau}$ respectively $\CC/\Lambda_{2\omega_1(\tau),
\omega_2(\tau)}$. Writing $\wp_{\tau}$ respectively
$\mathrm{cd}_{\tau}\hspace{-0.7mm} =\hspace{-0.7mm}
\mathrm{cn}/\mathrm{dn}$ for the Weierstrassian function on
$E_{\tau}$ respectively the Jacobian $\mathrm{cd}$ function on
$E_{\tau}'$, they are of order 2. There are natural cyclic isogenies
$[n]\mathpunct{:} E_{\tau}\hspace{-0.7mm} \to\hspace{-0.7mm}
E_{n\tau}$ and $[n]\mathpunct{:} E_{\tau}'\hspace{-0.7mm}
\to\hspace{-0.7mm} E_{n\tau}'$, and according to the theory of
descent we write $n_{\tau}$ and $\mathcal{T}_{n,\,\tau}$ for the
rational functions which make the following diagrams

\begin{minipage}{\textwidth}
\begin{minipage}[t]{0.5\textwidth}
\begin{displaymath}
\xymatrix{
\ \, E_{\tau} \ar[d]^{\wp_{\tau}}  \ar[r]^{[n]}& \ \, E_{n\tau} \ar[d]^{\wp_{n\tau}} \\
\ \, \PP^1   \ar[r]^{n_{\tau}}    &  \ \, \PP^1 }
\end{displaymath}
\end{minipage}
\hspace{-2cm}
\begin{minipage}[t]{0.5\textwidth}
\begin{displaymath}
\xymatrix{
\ \,E_{\tau}' \ar[d]^{\mathrm{cd}_{\tau}}  \ar[r]^{[n]}& \ \, E_{n\tau} \ar[d]^{\mathrm{cd}_{n\!\tau}}' \\
\ \, \PP^1   \ar[r]^{\mathcal{T}_{n,\,\tau}}    &  \ \, \PP^1 }
\end{displaymath}
\end{minipage}
\end{minipage}
commutative.

Henceforth an {\it elliptic} rational function refers to a
$f\hspace{-0.7mm} \in\hspace{-0.7mm} \mathrm{End}(\PP^1)$ that
satisfies $f\hspace{-0.7mm} \sim\hspace{-0.7mm} n_{\tau}$ in
$(\mathrm{End}(\PP^1), \circ)$ for some $(n, \tau)\hspace{-0.7mm}
\in\hspace{-0.7mm} \NN\mathopen \times\mathclose \HH.$ This notion
is general than the one used by engineers\,(cf.\,\cite{LTE01}). We
have that $\mathcal{T}_{n,\,\tau}$ is elliptic because
$\mathcal{T}_{n,\,\tau}\hspace{-0.7mm} \sim\hspace{-0.7mm}
n_{\tau/2}$, which will be called {\it generalized Zolotarev
fractions}. The principal result of this section is that
$\{\text{Elliptic rational functions of degree}\ n\hspace{-0.7mm}
\ge\hspace{-0.7mm} 3\}/\sim$ is $Y_0(n)$, and we begin with

\begin{lemma}\label{lemma for invariant}
If $\tau\hspace{-0.7mm} \in\hspace{-0.7mm} \HH$ and if
$n\hspace{-0.7mm} \ge\hspace{-0.7mm} 3$ then
\begin{eqnarray*}
\mathfrak{o}_{n_{\tau}} \hspace{-0.5mm} =\hspace{-0.5mm}
\wp_{n\tau}(E_{n\tau}[2]) \ \ \text{and} \ \
n_{\tau}\hspace{-2mm}^{-\!1}(\mathfrak{o}_{n_{\tau}})\mathopen\setminus\mathclose
{} |\mathfrak{O}_{n_{\tau}}| \hspace{-0.5mm} =\hspace{-0.5mm}
\wp_{\tau}(E_{\tau}[2]).
\end{eqnarray*}
\end{lemma}
\begin{proof}
This follows from a calculation of local ramification degree.
\end{proof}

Then we prove

\begin{theorem}\label{moduli space}
Given $\tau_1, \tau_2\hspace{-0.7mm} \in\hspace{-0.7mm} \HH$ and
given $n\hspace{-0.7mm} \ge\hspace{-0.7mm} 3$. Then $ n_{\tau_1}
\hspace{-0.7mm} \sim \hspace{-0.7mm} n_{\tau_2}$ in
$(\mathrm{End}(\PP^1), \circ)$ if and only if $\Gamma_0(n)
\tau_2\hspace{-0.5mm} =\hspace{-0.5mm} \Gamma_0(n) \tau_1$, where
\begin{eqnarray*}
\Gamma_0(n) \hspace{-0.5mm} =\hspace{-0.5mm} \Bigg\{\Bigg(
\begin{array}{cc}
  a     & b\\
  c & d
\end{array}\Bigg)
\hspace{-0.5mm} \in \mathrm{SL}_2(\ZZ)\, \Bigg|\Bigg.
\vphantom{\begin{pmatrix}
  a     & b\\
  c & d
\end{pmatrix}} c\hspace{-0.5mm} \equiv\hspace{-0.5mm} 0 \hspace{-2mm}
\pmod{n}\Bigg\}
\end{eqnarray*}
is the modular group.
\end{theorem}
\begin{proof}
Write $e_i(\tau)$ for $e_i$ with respect to the pair of primitive
periods $(1, \tau)$. First of all we show that for any pair $(n,
\tau)\hspace{-0.7mm} \in\hspace{-0.7mm} \NN\mathopen
\times\mathclose \HH$ and $0 \hspace{-0.7mm} \leq \hspace{-0.7mm} i
\hspace{-0.7mm} \leq \hspace{-0.7mm} 3$ there exist $(\iota,
\epsilon)\hspace{-0.7mm} \subset\hspace{-0.7mm}\mathrm{Aut}(\PP^1)$
such that $n_{\tau} \hspace{-0.7mm} =\hspace{-0.7mm}
\epsilon\hspace{-0.7mm} \circ\hspace{-0.7mm} n_{\tau}\hspace{-0.7mm}
\circ\hspace{-0.7mm} \iota\!^{-\!1}$ and
$\iota(e_i(\tau))\hspace{-0.7mm} =\hspace{-0.7mm} e_0(\tau)$. We
only verify this claim for $i\hspace{-0.5mm} =\hspace{-0.5mm} 1$
since similar arguments apply to other situations. The map
$\overline{\iota}\mathpunct{:} E_{\tau} \hspace{-0.7mm}
\to\hspace{-0.7mm} E_{\tau}$ defined by
$\overline{\iota}(z)\hspace{-0.7mm} =\hspace{-0.7mm}
z\hspace{-0.7mm} +\hspace{-0.7mm} 1/2$ descends to
$\iota\hspace{-0.7mm} \in\hspace{-0.7mm}\mathrm{Aut}(\PP^1)$ with
respect to $\wp_{\tau}$, and the map
$\overline{\epsilon}\mathpunct{:} E_{n\tau}\hspace{-0.7mm}
\to\hspace{-0.7mm} E_{n\tau}$ given by
$\overline{\epsilon}(w)\hspace{-0.7mm} =\hspace{-0.7mm}
w\hspace{-0.7mm} +\hspace{-0.7mm} n/2$ descends to
$\epsilon\hspace{-0.7mm} \in\hspace{-0.7mm}\mathrm{Aut}(\PP^1)$ with
respect to $\wp_{n\tau}$.
\begin{displaymath}
\xymatrix@!0{
  & E_{\tau} \ar[rr]^{[n]} \ar'[d][dd]
      &  & E_{n\tau} \ar[dd]       \\
  E_{\tau} \ar[ur]^{\overline{\iota}}\ar[rr]_{[n]\ \ \, \ \  }\ar[dd]
      &  & E_{n\tau} \ar[ur]_{\overline{\epsilon} }\ar[dd]\\
  & \PP^1 \ar'[r]^{ \ \ \  n_{\tau}}[rr]
      &  & \PP^1                \\
  \PP^1 \ar[rr]^{n_{\tau}}\ar[ur]^{\iota}
      &  & \PP^1 \ar[ur]^{\epsilon}        }
\end{displaymath}
One checks easily that $\epsilon\!^{-\!1}\hspace{-0.7mm}
\circ\hspace{-0.7mm} n_{\tau}\hspace{-0.7mm} \circ\hspace{-0.7mm}
\iota\hspace{-0.7mm} =\hspace{-0.7mm} n_{\tau}$ and
$\iota(e_0)\hspace{-0.7mm} =\hspace{-0.7mm} e_1$ which proves the
desired claim.

By construction we have $n_{\tau_i}\hspace{-0.7mm}
\circ\hspace{-0.7mm} \wp_{\tau_i}\hspace{-0.7mm} =\hspace{-0.7mm}
\wp_{n \tau_i}\hspace{-0.7mm} \circ\hspace{-0.7mm} [n]$ where $[n]$
maps $E_{\tau_i}$ to $E_{n\tau_i}$ for
$1\hspace{-0.7mm}\leq\hspace{-0.7mm}
i\hspace{-0.7mm}\leq\hspace{-0.7mm} 2$. If there exist $\{\epsilon,
\varepsilon\}\hspace{-0.7mm}
\subset\hspace{-0.7mm}\mathrm{Aut}(\PP^1)$ such that
$\epsilon\hspace{-0.7mm} \circ\hspace{-0.7mm}
n_{\tau_1}\hspace{-0.7mm} \circ\hspace{-0.7mm}
\varepsilon\!^{-\!1}\hspace{-0.7mm} =\hspace{-0.5mm} n_{\tau_2}$
then $\epsilon$ induces a bijection between
$\wp_{n\tau_1}(E_{n\tau_1}[2])$ and $\wp_{n\tau_2}(E_{n\tau_2}[2])$,
because $\mathfrak{o}_{n_{\tau_{i}}}\hspace{-0.7mm} =\hspace{-0.7mm}
\wp_{n\tau_i}(E_{n\tau_i}[2])$ as $n\hspace{-0.7mm}
\ge\hspace{-0.7mm} 3$. Moreover $\varepsilon\!^{-\!1}$ induces a
bijection between
$n\!^{-\!1}_{\tau_2}(\mathfrak{o}_{n_{\tau_2}})$\,(resp.\,$|\mathfrak{O}_{n_{\tau_2}}|$)
and
$n\!^{-\!1}_{\tau_1}(\mathfrak{o}_{n_{\tau_1}})$\,(resp.\,$|\mathfrak{O}_{n_{\tau_1}}|$),
and then we deduce from Lemma \ref{lemma for invariant} that
$\varepsilon\!^{-\!1}$ also induces a bijection between
$\wp_{\tau_2}(E_{\tau_2}[2])$ and $\wp_{\tau_1}(E_{\tau_1}[2])$. The
monodromy representation of a small loop around any critical value
of $\wp$ is an involution, and consequently the map
$\varepsilon$\,(resp.\,$\epsilon$)$\mathpunct{:}
\PP^1\hspace{-0.7mm} \to\hspace{0.7mm} \PP^1$ lifts to an
isomorphism $\overline{\varepsilon}\mathpunct{:}
E_{\tau_1}\hspace{-0.7mm}\to\hspace{-0.7mm}E_{\tau_2}$\,(resp.\,$\overline{\epsilon}\mathpunct{:}
E_{n\tau_1}\hspace{-0.7mm}\to\hspace{-0.7mm}E_{n\tau_2}$) such that
$\wp_{\tau_2} \hspace{-0.7mm} \circ\hspace{-0.7mm}
\overline{\varepsilon} \hspace{-0.7mm} =\hspace{-0.7mm} \varepsilon
\hspace{-0.7mm} \circ\hspace{-0.7mm}
\wp_{\tau_1}$\,(resp.\,$\wp_{n\tau_2} \hspace{-0.7mm}
\circ\hspace{-0.7mm} \overline{\epsilon} \hspace{-0.7mm}
=\hspace{-0.7mm} \epsilon \hspace{-0.7mm} \circ\hspace{-0.7mm}
\wp_{n\tau_1}$). By the claim made in the previous paragraph we may
assume
$\varepsilon\!^{-\!1}(e_0(\tau_2))\hspace{-0.7mm}=\hspace{-0.7mm}
e_0(\tau_1)$, hence $\overline{\varepsilon}(0)\hspace{-0.7mm}
=\hspace{-0.7mm} 0$ and
$\overline{\varepsilon}^{-\!1}(z)\hspace{-0.7mm} =\hspace{-0.7mm}
\gamma z$ with $\gamma\hspace{-0.7mm} \in\hspace{-0.7mm} \CC^*$ and
with $\overline{\varepsilon}^{-\!1}$ giving a bijection between
$\Lambda_{1, \,
\tau_2}$(resp.\,$[n]\!^{-\!1}(E_{n\tau_2}[2])\hspace{-0.7mm}
=\hspace{-0.7mm} \Lambda_{\frac{1}{2n}, \frac{\tau_2}{2}}$) and
$\Lambda_{1, \,
\tau_1}$(resp.\,$[n]\!^{-\!1}(E_{n\tau_1}[2])\hspace{-0.7mm}
=\hspace{-0.7mm} \Lambda_{\frac{1}{2n}, \frac{\tau_1}{2}}$). Writing
$\gamma \tau_2\hspace{-0.7mm} =\hspace{-0.7mm} a
\tau_1\hspace{-0.7mm} +\hspace{-0.7mm} b$ and $\gamma\hspace{-0.7mm}
=\hspace{-0.7mm} c \tau_1\hspace{-0.7mm} +\hspace{-0.7mm} d$ with
 $\begin{scriptsize}\begin{pmatrix}
  a     & b\\
  c & d
\end{pmatrix}\end{scriptsize}\hspace{-0.7mm}\in\hspace{-0.7mm} SL_2(\ZZ)$, by using $\gamma \Lambda_{\frac{1}{2n},
\frac{\tau_2}{2}}\hspace{-0.7mm} =\hspace{-0.7mm}
\Lambda_{\frac{1}{2n}, \frac{\tau_1}{2}}$ we have
$\frac{c\tau_1\mathopen +\mathclose d}{2n}\hspace{-0.7mm}
\in\hspace{-0.7mm} \Lambda_{\frac{1}{2n}, \frac{\tau_1}{2}}$ and
therefore $n|c$. This verifies that
$\begin{scriptsize}\begin{pmatrix}
  a     & b\\
  c & d
\end{pmatrix}\end{scriptsize}\hspace{-0.7mm} \in\hspace{-0.7mm}
\Gamma_0(n)$.

It remains to check $n_{\tau_2}\hspace{-0.7mm} \sim\hspace{-0.7mm}
n_{\tau_1}$ when $\tau_2\hspace{-0.7mm} = \hspace{-0.7mm}
\begin{scriptsize}\begin{pmatrix}
  a     & b\\
  c & d
\end{pmatrix}\end{scriptsize}\tau_1$ with  $\begin{scriptsize}\begin{pmatrix}
  a     & b\\
  c & d
\end{pmatrix}\end{scriptsize}\hspace{-0.7mm}\in\hspace{-0.7mm}  \Gamma_0(n)$.
\begin{displaymath}
\xymatrix@!0{
  & E_{\tau_2} \ar[rr]^{[n]} \ar'[d][dd]
      &  & E_{n\tau_2} \ar[dd]       \\
  E_{\tau_1} \ar[ur]^{\overline{\varepsilon}}\ar[rr]_{[n]\ \ \, \ \  }\ar[dd]
      &  & E_{n\tau_1} \ar[ur]_{\overline{\epsilon} }\ar[dd]\\
  & \PP^1 \ar'[r]^{ \ \ \  n_{\tau_2}}[rr]
      &  & \PP^1                \\
  \PP^1 \ar[rr]^{n_{\tau_1}}\ar[ur]^{\varepsilon}
      &  & \PP^1 \ar[ur]^{\epsilon}        }
\end{displaymath}
Set $\gamma\hspace{-0.7mm} =\hspace{-0.7mm} c \tau_1\hspace{-0.7mm}
+\hspace{-0.7mm} d$ then the map
$\overline{\varepsilon}\mathpunct{:} z\hspace{-0.7mm}
\in\hspace{-0.7mm} E_{\tau_1}\hspace{-0.7mm} \mapsto\hspace{-0.7mm}
z/\gamma\hspace{-0.7mm} \in\hspace{-0.7mm} E_{\tau_2}$ is an
isomorphism and descends to $\varepsilon\hspace{-0.7mm}
\in\hspace{-0.7mm} \mathrm{Aut}(\PP^1)$ in the sense that
$\wp_{\tau_2}\hspace{-0.7mm} \circ\hspace{-0.7mm}
\overline{\varepsilon}\hspace{-0.7mm} =\hspace{-0.7mm}
\varepsilon\hspace{-0.7mm} \circ\hspace{-0.7mm} \wp_{\tau_1}$.
Moreover $\overline{\epsilon}\mathpunct{:} z\hspace{-0.7mm}
\in\hspace{-0.7mm} E_{n\tau_1} \hspace{-0.7mm}
\mapsto\hspace{-0.7mm} z/\gamma\hspace{-0.7mm} \in\hspace{-0.7mm}
E_{n\tau_2}$ is also an isomorphism\,(here we use $n|c$) and
descends to $\epsilon\hspace{-0.7mm} \in\hspace{-0.7mm}
\mathrm{Aut}(\PP^1)$ in the sense that $\wp_{n\tau_2}\hspace{-0.7mm}
\circ\hspace{-0.7mm} \overline{\epsilon}\hspace{-0.7mm}
=\hspace{-0.7mm} \epsilon\hspace{-0.7mm} \circ\hspace{-0.7mm} \wp_{n
\tau_1}$. One checks readily that $\epsilon\hspace{-0.7mm}
\circ\hspace{-0.7mm} n_{\tau_1}\hspace{-0.7mm} =\hspace{-0.7mm}
n_{\tau_2}\hspace{-0.7mm} \circ\hspace{-0.7mm} \varepsilon$.
\end{proof}

In \cite{NW08} we have proved that
\begin{eqnarray}\label{T=T}
\mathcal{T}_{n,\, t}(z) = \sqrt{k(4nti/\pi)} \mathcal{T}_{n,\,
4ti/\pi}(z/\sqrt{k(4ti/\pi)}).
\end{eqnarray}

By using (\ref{T=T}), Theorem \ref{moduli space} and the injectivity
of $i\,\RR_{>0}\hspace{-0.7mm} \hookrightarrow\hspace{-0.7mm}
\Gamma_0(n)\backslash \HH$ we have for $t_1,t_2 > 0$

\begin{corollary}\label{linearequivalence}
If $n\hspace{-0.7mm} \ge\hspace{-0.7mm} 3$ then
$\mathcal{T}_{n,\,t_1} \hspace{-0.7mm} \sim \hspace{-0.7mm}
\mathcal{T}_{n,\,t_2}$ in $(\mathrm{End}(\PP^1), \circ)$ if and only
if $t_1\hspace{-0.7mm}=\hspace{-0.7mm} t_2$.
\end{corollary}

We shall indicate that Theorem \ref{Rittsecond} is applicable to all
elliptic rational functions.
\begin{lemma}
Let $f\mathpunct{:} M\hspace{-0.6mm}\to\hspace{-0.6mm}N$ be a finite
map and let $\alpha$ be a closed cycle on $N$ over which $f$ is
unramified. If $f\!^{-\!1}(\alpha)$ is connected then the monodromy
action of $\alpha$ is transitive.
\end{lemma}
\begin{proof}
It is almost the definition.
\end{proof}

We write $C_{\tau}$ for the Jordan curve on $\PP^1$ which is given
by $\wp_{\tau}(\{z: \Im{z} =\Im{\tau}/4 \})$.

\begin{proposition}\label{elliptic is transitive}
Given $\tau \hspace{-0.7mm} \in \hspace{-0.7mm} \HH$ and given
$n\hspace{-0.7mm} \in \hspace{-0.7mm} \NN$, there exists a closed
cycle $\alpha$ on $\PP^1$, along which $n_{\tau}$ is unramified,
such that its monodromy action is transitive.
\end{proposition}
\begin{proof}
By definition we have $n^{-\!1}_{\tau}(C_{n\tau}) = C_{\tau}$, and
our previous lemma applies.
\end{proof}

The nesting property of Zolotarev's fractions are important in
engineering, and for general elliptic rational functions we have
\begin{proposition}[Nesting Property]\label{nesting general}
Given $m,n\hspace{-0.7mm} \in\hspace{-0.7mm} \NN$, $\tau
\hspace{-0.7mm} \in \hspace{-0.7mm} \HH$ and $t \hspace{-0.7mm} >
\hspace{-0.7mm} 0$ we have  $ (mn)_{\tau}
\hspace{-0.7mm}=\hspace{-0.7mm} m_{n\tau}\hspace{-0.7mm}
\circ\hspace{-0.7mm} n_{\tau}, \mathcal{T}_{mn,\,\tau}
\hspace{-0.7mm}=\hspace{-0.7mm}
\mathcal{T}_{m,\,n\tau}\hspace{-0.7mm} \circ\hspace{-0.7mm}
\mathcal{T}_{n,\,\tau}$ and $\mathcal{T}_{mn,\,t}
\hspace{-0.7mm}=\hspace{-0.7mm} \mathcal{T}_{m,\,nt}\hspace{-0.7mm}
\circ\hspace{-0.7mm} \mathcal{T}_{n,\,t}. $
\end{proposition}

One checks easily that $f\hspace{-0.7mm} \in\hspace{-0.7mm}
\mathrm{End}(\EE)$ is elliptic if and only if it is a
Chebyshev-Blaschke product. For any elliptic $f\hspace{-0.7mm}
\in\hspace{-0.7mm} \mathrm{End}(\EE)$ there exists $t\hspace{-0.7mm}
>\hspace{-0.7mm} 0$ such that $f \hspace{-0.7mm}\sim \hspace{-0.7mm}\mathcal{T}_{n, \, t}$ in
$(\mathrm{End}(\EE),\circ)$. We set $\chi(f)\hspace{-0.7mm}
=\hspace{-0.7mm} nt$ when $f$ is of degree at least three, which is
well-defined by Theorem \ref{moduli space} and will be called the
{\it moduli} of $f$.

\section{Rigidity of monoid factorizations}\label{sectionrigidity}
The main result of \cite{NW08} implicitly gives generators of
relations of $(\mathrm{End}(\EE), \circ)$.

\begin{theorem}[Ng-Wang]\label{NgWangmonoid}
The monoid $(\mathrm{End}(\EE), \circ)$ is presented by $\left<S
\,|\, R \right>$ where $S$ consists of linear and of prime finite
Blaschke product and $R$ consists of
\begin{enumerate}
\item[(i)] $\iota\hspace{-0.7mm} \circ\hspace{-0.7mm} f\hspace{-0.7mm} =\hspace{-0.7mm} g$ or $f \hspace{-0.7mm} \circ\hspace{-0.7mm}\iota \hspace{-0.7mm}
=\hspace{-0.7mm} g$ where $\iota \hspace{-0.7mm} \in\hspace{-0.7mm}
\mathrm{Aut}(\EE)$;
\item[(ii)] $z^r g(z)^k\hspace{-0.7mm}  \circ\hspace{-0.7mm}
z^k\hspace{-0.7mm} =\hspace{-0.7mm}  z^k\hspace{-0.7mm}
\circ\hspace{-0.7mm} z^rg(z^k)$ with $(k, r)\hspace{-0.7mm}
=\hspace{-0.7mm} 1$;
\item[(iii)]  $ \mathcal{T}_{p,\, qt}\hspace{-0.7mm} \circ\hspace{-0.7mm}
\mathcal{T}_{q,\, t}\hspace{-0.7mm} =\hspace{-0.7mm}
\mathcal{T}_{q,\, pt}\hspace{-0.7mm} \circ\hspace{-0.7mm}
\mathcal{T}_{p, \, t}$ with $p, q$ primes and $t$ a positive real
number.
\end{enumerate}
\end{theorem}

We call a relation $ a\mathopen \circ\mathclose b\hspace{-0.7mm}
=\hspace{-0.7mm} c\mathopen \circ\mathclose d$ with
$\deg{a}\hspace{-0.7mm} =\hspace{-0.7mm} \deg{d}$ and $(\deg{a},
\deg{b})\hspace{-0.7mm}  = \hspace{-0.7mm} 1$ in terms of
irreducible\;(resp. not necessary irreducible) elements a {\it
Ritt\,(resp.\,generalized Ritt) relation} of $(\mathrm{End}(X),
\circ)$. Presentations of Monoids in Theorem \ref{NgWangmonoid}
involves only Ritt relations. The next result also follows from
\cite{NW08}.

\begin{theorem}[Ng-Wang]\label{Ritt move for Blaschke}
If $a\mathopen \circ\mathclose b\hspace{-0.7mm}  =\hspace{-0.7mm}
c\mathopen \circ\mathclose d$ is a generalized Ritt relation in
$(\mathrm{End}(\EE), \circ)$ then up to units of
$(\mathrm{End}(\EE), \circ)$ and up to the permutation
$a\hspace{-0.8mm} \leftrightarrow\hspace{-0.8mm} c, b\hspace{-0.8mm}
\leftrightarrow\hspace{-0.8mm} d$ we are in the case $z^s
g(z)^n\hspace{-0.7mm}  \circ\hspace{-0.7mm} z^n\hspace{-0.7mm}
=\hspace{-0.7mm}  z^n\hspace{-0.7mm} \circ\hspace{-0.7mm}
z^sg(z^n)((n, s)\hspace{-0.7mm} =\hspace{-0.7mm} 1)$ or
$\mathcal{T}_{m,\, nt}\hspace{-0.7mm} \circ\hspace{-0.7mm}
\mathcal{T}_{n,\, t}\hspace{-0.7mm} =\hspace{-0.7mm}
\mathcal{T}_{n,\, mt}\hspace{-0.7mm} \circ\hspace{-0.7mm}
\mathcal{T}_{m, \, t}((m, n)\hspace{-0.7mm} =\hspace{-0.7mm} 1,
t\hspace{-0.7mm}
>\hspace{-0.7mm} 0)$.
\end{theorem}

We call $f\hspace{-0.7mm} \in\hspace{-0.7mm}\mathrm{End}(\EE)$ {\it
totally ramified} if $f\hspace{-0.7mm} \sim\hspace{-0.7mm} z^{n}$ in
$(\mathrm{End}(\mathbb{E}), \circ)$. The following simple remark is
a complement of the above theorem.
\begin{lemma}\label{bloom and totally ramified}
Let $h\hspace{-0.7mm} \in\hspace{-0.7mm} \mathrm{End}(\EE)$ satisfy
$h(0)\hspace{-0.7mm} \neq\hspace{-0.7mm} 0$ and let $\{s,
n\}\hspace{-0.7mm} \subset\hspace{-0.7mm} \NN$ satisfy
$n\hspace{-0.7mm} \ge\hspace{-0.7mm} 2$. Then neither $z^sh(z)^n$
nor $z^sh(z^n)$ is totally ramified.
\end{lemma}

In this section we will prove, via action of fundamental groups,
some rigidity properties of factorizations of $(\mathrm{End}(\EE),
\circ)$. The following generalizes a result of \cite{ZM10}.

\begin{proposition}\label{rigidity}
Let $f\mathpunct{:} M\hspace{-0.7mm} \to\hspace{-0.7mm} N$ be a
finite map of degree $n$, $q\hspace{-0.7mm} \in\hspace{-0.7mm}
N\mathopen \setminus\mathclose \mathfrak{d}_f$ and $\alpha
\hspace{-0.7mm} \in \hspace{-0.7mm} \pi_{1} (
N\mathopen\setminus\mathclose \mathfrak{d}_f , q)$. If finite maps
$b\mathpunct{:} M\hspace{-0.7mm} \to\hspace{-0.7mm}A$,
$a\mathpunct{:} A\hspace{-0.7mm} \to\hspace{-0.7mm}N$,
$d\mathpunct{:} M\hspace{-0.7mm} \to\hspace{-0.7mm}R$,
$c\mathpunct{:} R\hspace{-0.7mm} \to\hspace{-0.7mm}N$ satisfy
$a\hspace{-0.7mm} \circ\hspace{-0.7mm} b\hspace{-0.7mm}
=\hspace{-0.7mm} c\hspace{-0.7mm} \circ\hspace{-0.7mm}
d\hspace{-0.7mm} =\hspace{-0.7mm} f$ and if the monodromy action of
$\alpha$ is transitive then there exist Riemann surfaces $T, W$ and
finite maps $h\mathpunct{:}M \hspace{-0.7mm} \to\hspace{-0.7mm} T,
\overline{b}\mathpunct{:} T \hspace{-0.7mm} \to\hspace{-0.7mm} A,
\overline{d}\mathpunct{:}T \hspace{-0.7mm} \to\hspace{-0.7mm} R,
\overline{a} \mathpunct{:}A \hspace{-0.7mm} \to\hspace{-0.7mm} W,
\overline{c}\mathpunct{:}R \hspace{-0.7mm} \to\hspace{-0.7mm} W,
g\mathpunct{:}W \hspace{-0.7mm} \to\hspace{-0.7mm} N$ such that $
\deg{g}\hspace{-0.7mm} =\hspace{-0.7mm} (\deg{a}, \deg{c}),
\deg{h}\hspace{-0.7mm} =\hspace{-0.7mm} (\deg{b}, \deg{d})$ and the
following diagram \vspace{-3mm}
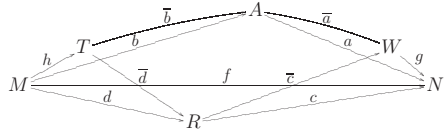
\begin{center}{\Huge
\scalebox{.3}{\xymatrix{
    &  &     &  & & &  & &   A \ar@/^0.5pc/[drrrrr]^{\overline{a}} \ar[ddrrrrrr]^{a} & & & &  &    &\\
    &  &  T \ar@/^0.5pc/[urrrrrr]^{\overline{b}} \ar[ddrrrr]^{\overline{d}}  &  & &  & & &    & & & &  & W \ar[dr]^{g}  & \\
 M \ar[rrrrrrrrrrrrrr]^{f}\ar[drrrrrr]^{d} \ar[urr]^{h} \ar[uurrrrrrrr]^{b}  &       &  & & & & & &    & & & &  &    & N \\
    &  &    & & & &R \ar[urrrrrrrr]^{c} \ar[uurrrrrrr]^{\overline{c}} & &    & & & &  &    &
    }}}
\end{center}
commutates.
\end{proposition}

\begin{proof}
By Theorem \ref{Rittsecond} the lattice of groups intermediate
between $\pi_1(N\mathopen \setminus\mathclose \mathfrak{d}_f)$ and
$\pi_1(M\mathopen \setminus\mathclose f\!^{-\!1}(\mathfrak{d}_f))$
is isomorphic to a sublattice of $\mathcal{L}_n$, and by Lemma
\ref{S} it suffices to verify the following: if $\calL$ is a
sublattice of $(\calL_n ; \le)$ and contains $a$ and $b$ then it
also contains $(a, b)$ and $[a,b]$. Indeed this follows immediately
from the definition of sublattice and it can be illustrated by the
following figure \vspace{-3mm}
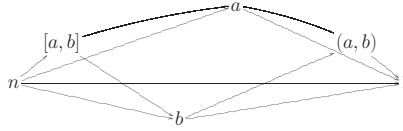
\begin{center}{\Huge
\scalebox{.3}{\xymatrix{
    &  &     &  &  &  & &   a \ar@/^0.5pc/[drrrr] \ar[ddrrrrr] & & & &      &\\
    & [a, b] \ar@/^0.5pc/[urrrrrr] \ar[ddrrrr] &   &   &  & &    & & & &  & (a, b) \ar[dr]  & \\
 n \ar[rrrrrrrrrrrr]\ar[drrrrr] \ar[ur] \ar[uurrrrrrr]  &         & &  & & &    & & & &  &    & 1 \\
    &  &     & & &b \ar[urrrrrrr] \ar[uurrrrrr] & &    & & & &      &
    }}}
\end{center}
where we use $s\hspace{-0.7mm} \to\hspace{-0.7mm} t$ to denote
$s\hspace{-0.7mm} \le\hspace{-0.7mm} t$ (for lattice structure) or
equivalently $t | s$.
\end{proof}

Proposition \ref{rigidity} applies to finite Blaschke products and
gives

\begin{proposition}\label{rigidity of Blaschke}
Let $a, b, c, d, f$ be finite Blaschke products that satisfy
$a\hspace{-0.7mm} \circ\hspace{-0.7mm} b\hspace{-0.7mm}
=\hspace{-0.7mm} c\hspace{-0.7mm} \circ\hspace{-0.7mm}
d\hspace{-0.7mm} =\hspace{-0.7mm} f$. There exist $\{\ol{a\bPh},
\overline{b}, \overline{c\bPh}, \overline{d}, h, g\}\hspace{-0.7mm}
\subset\hspace{-0.7mm} \mathrm{End}(\EE)$ such that
\begin{enumerate}
\item[(i)] $g\hspace{-0.7mm} \circ\hspace{-0.7mm}
\overline{a}\hspace{-0.7mm} =\hspace{-0.7mm} a, \ g\hspace{-0.7mm}
\circ\hspace{-0.7mm} \overline{c}\hspace{-0.7mm} =\hspace{-0.7mm} c,
\ \deg{g}\hspace{-0.7mm} =\hspace{-0.7mm} (\deg{a}, \deg{c})$;
\item[(ii)] $\overline{b}\hspace{-0.7mm} \circ\hspace{-0.7mm}
h\hspace{-0.7mm} =\hspace{-0.7mm} b, \ \overline{d}\hspace{-0.7mm}
\circ\hspace{-0.7mm} h\hspace{-0.7mm} =\hspace{-0.7mm} d, \
\deg{h}\hspace{-0.7mm} =\hspace{-0.7mm} (\deg{b}, \deg{d})$;
\item[(iii)] $\overline{a\bPh}\hspace{-0.7mm} \circ\hspace{-0.7mm}
\overline{b}\hspace{-0.7mm} =\hspace{-0.7mm}
\overline{c\bPh}\hspace{-0.7mm} \circ\hspace{-0.7mm} \overline{d}$.
\end{enumerate}
\end{proposition}
\begin{proof}
We regard these finite Blaschke products as finite maps $\EE
\hspace{-0.7mm} \to \hspace{-0.7mm} \EE$, for which the monodromy
action of any loop closely around the unit circle are transitive. By
Lemma \ref{disk and complex plane} the decomposition of finite
Blaschke products into finite Blaschke products is essentially
equivalent to that of finite Blaschke products into finite maps. Now
we may apply Proposition \ref{rigidity} directly to deduce the
desired assertion.
\end{proof}

We give a simple example to explain how the above rigidity applies.

\begin{corollary}\label{simple rigidity}
Let $f\mathpunct{:} M\hspace{-0.7mm} \to\hspace{-0.7mm} N$ be a
finite map that satisfies the monodromy condition required in
Proposition \ref{rigidity}. If there are decompositions of $f$ into
finite maps $f\hspace{-0.7mm} = \hspace{-0.7mm} a\hspace{-0.7mm}
\circ\hspace{-0.7mm} b \hspace{-0.7mm} = \hspace{-0.7mm} c
\hspace{-0.7mm} \circ \hspace{-0.7mm} d$ with
$\deg{a}\hspace{-0.7mm}=\hspace{-0.7mm} \deg{c}$, then there exist
biholomorphic maps $\iota$ such that
\begin{eqnarray*}
a \hspace{-0.7mm} = \hspace{-0.7mm} c \hspace{-0.7mm} \circ
\hspace{-0.7mm} \iota^{-\!1}, \ \ \ b \hspace{-0.7mm} =
\hspace{-0.7mm} \iota \hspace{-0.7mm} \circ \hspace{-0.7mm} d.
\end{eqnarray*}
\end{corollary}

\begin{proof}
Applying Proposition \ref{rigidity} we obtain suitable $\ol{a\bPh},
\overline{b}, \overline{c\bPh}, \overline{d}, h$ and $g$. Because
$\deg{a}\hspace{-0.7mm}=\hspace{-0.7mm} \deg{c}$ and
$\deg{b}\hspace{-0.7mm}=\hspace{-0.7mm} \deg{d}$ it is clear that
$\ol{a\bPh}, \overline{b}, \overline{c\bPh}, \overline{d}$ are all
biholomorphic maps. One may choose
$\iota\hspace{-0.7mm}=\hspace{-0.7mm} \overline{a}^{-\!1}
\hspace{-0.7mm} \circ \hspace{-0.7mm} \overline{c}$ to fulfill the
desired assertion.
\end{proof}

For totally ramified maps we have
\begin{corollary}\label{special ramified}
If $f\hspace{-0.7mm}\in\hspace{-0.7mm}\mathrm{End}(\EE)$ is of
degree $s\hspace{-0.7mm} \ge \hspace{-0.7mm} 2$ and if $f^t$ is
totally ramified for some integer $t\hspace{-0.7mm}\ge
\hspace{-0.7mm} 2$, then there exists $p\hspace{-0.7mm}\in
\hspace{-0.7mm} \EE$ and $\rho \hspace{-0.7mm}\in \hspace{-0.7mm}
\TT$ such that $f \hspace{-0.7mm}=  \hspace{-0.7mm} \iota_{p}
\hspace{-0.7mm}\circ \hspace{-0.7mm} \rho z^s \hspace{-0.7mm}\circ
\hspace{-0.7mm} \iota_{-\!p} $.
\end{corollary}
\begin{proof}
By assumption there exist $\{\epsilon, \varepsilon\} \hspace{-0.7mm}
\subset \hspace{-0.7mm} \mathrm{Aut}(\EE)$ such that $f^t
\hspace{-0.7mm}= \hspace{-0.7mm} \epsilon \hspace{-0.7mm}\circ
\hspace{-0.7mm} z^{s^t} \hspace{-0.7mm}\circ \hspace{-0.7mm}
\varepsilon$ which gives
\begin{eqnarray*}
f \hspace{-0.7mm}\circ \hspace{-0.7mm}f^{t\mathopen-\mathclose 1}
\hspace{-0.7mm}= \hspace{-0.7mm}(\epsilon \hspace{-0.7mm}\circ
\hspace{-0.7mm} z^{s})\hspace{-0.7mm}\circ \hspace{-0.7mm}
(z^{s^{t\mathopen-\mathclose 1}} \hspace{-0.7mm}\circ
\hspace{-0.7mm} \varepsilon).
\end{eqnarray*}
This together with Corollary \ref{simple rigidity} implies that
$f\hspace{-0.7mm}\sim\hspace{-0.7mm}\epsilon \hspace{-0.7mm}\circ
\hspace{-0.7mm} z^{s}$ in $(\mathrm{End}(\EE), \circ)$ and therefore
$f$ is totally ramified. Writing $p\hspace{-0.7mm} =\hspace{-0.7mm}
\mathfrak{d}_f$ and $q\hspace{-0.7mm} =\hspace{-0.7mm}
|\mathfrak{D}\!_f|$ we have $p \hspace{-0.7mm}= \hspace{-0.7mm} q$,
otherwise $f^t$ fails to be totally ramified. This gives readily $f
\hspace{-0.7mm}= \hspace{-0.7mm} \iota_{p} \hspace{-0.7mm}\circ
\hspace{-0.7mm} \rho z^s \hspace{-0.7mm}\circ \hspace{-0.7mm}
\iota_{-p} $ for some $\rho \hspace{-0.7mm}\in \hspace{-0.7mm} \TT$.
\end{proof}

For Chebyshev-Blaschke products we have
\begin{corollary}\label{special cheby}
If $f\hspace{-0.7mm} \in\hspace{-0.7mm} \mathrm{End}(\EE)$ is of
degree $s\hspace{-0.7mm}\ge \hspace{-0.7mm} 2$ and if
$n\hspace{-0.7mm} \ge \hspace{-0.7mm} 3$ then $f^{n}$ is not
elliptic.
\end{corollary}
\begin{proof}
If there exist $\{\epsilon, \varepsilon\}\hspace{-0.7mm}
\subset\hspace{-0.7mm} \mathrm{Aut}(\PP^1)$ and $t\hspace{-0.7mm}
>\hspace{-0.7mm}0$ such that $f^{n}\hspace{-0.7mm} =\hspace{-0.7mm}
\epsilon\hspace{-0.7mm} \circ\hspace{-0.7mm} \calT_{s^n,\,
t}\hspace{-0.7mm} \circ\hspace{-0.7mm} \varepsilon$ then it follows
from Proposition \ref{nesting general} that
\begin{eqnarray*}
f^2\hspace{-0.7mm} \circ\hspace{-0.7mm} f^{n \mathopen-\mathclose
2}\hspace{-2.7mm} &=&\hspace{-2.7mm} ( \epsilon\hspace{-0.7mm}
\circ\hspace{-0.7mm} \calT_{s^2,\, s^{n\mathopen-\mathclose 2}t}
)\hspace{-0.7mm} \circ\hspace{-0.7mm}
(\calT_{s^{n\mathopen-\mathclose 2}\!,\, t} \hspace{-0.7mm}
\circ\hspace{-0.7mm} \varepsilon)\\
f^{n \mathopen-\mathclose 2} \hspace{-0.7mm} \circ f^2
\hspace{-2.7mm} &=&\hspace{-2.7mm} ( \epsilon\hspace{-0.7mm}
\circ\hspace{-0.7mm} \calT_{s^{n\mathopen-\mathclose2},\, s^{2}t}
)\hspace{-0.7mm} \circ\hspace{-0.7mm} (\calT_{s^{2},\, t}
\hspace{-0.7mm} \circ\hspace{-0.7mm} \varepsilon).
\end{eqnarray*}
Proposition \ref{elliptic is transitive} enables us to apply
Corollary \ref{simple rigidity} to $f^n$ and obtain that $f^2$ is
associated to both $\calT_{s^2,\, s^{n\mathopen-\mathclose2}t}$ and
$\calT_{s^2,\, t}$. This leads to $\calT_{s^2,\,
s^{n\mathopen-\mathclose2}t}\hspace{-0.7mm}
\sim\hspace{-0.7mm}\calT_{s^2,\, t}$ which contradicts to Corollary
\ref{linearequivalence}, because $s^{n\mathopen-\mathclose2}t$ is
greater than $t$.
\end{proof}

\begin{corollary}\label{elliptic factor elliptic}
Let $f$ be an elliptic rational function and let $f\hspace{-0.7mm}
=\hspace{-0.7mm} a\hspace{-0.7mm} \circ\hspace{-0.7mm} b$ be a
relation in $(\mathrm{End}(\PP^1), \circ)$. Then $a$ and $b$ are
both elliptic.
\end{corollary}
\begin{proof}
Let $m\hspace{-0.7mm} =\hspace{-0.7mm} \deg{a}$ and let $n
\hspace{-0.7mm} =\hspace{-0.7mm} \deg{b}$. There exist $\{\epsilon,
\varepsilon\}\hspace{-0.7mm} \subset\hspace{-0.7mm}
\mathrm{Aut}(\PP^1)$ and $\tau \hspace{-0.7mm}\in\hspace{-0.7mm}
\HH$ such that $f\hspace{-0.7mm} =\hspace{-0.7mm}
\epsilon\hspace{-0.7mm} \circ\hspace{-0.7mm}
(mn)_{\tau}\hspace{-0.7mm} \circ\hspace{-0.7mm} \varepsilon$, and it
follows from the nesting property Proposition \ref{nesting general}
that
\begin{eqnarray*}
a \hspace{-0.7mm} \circ\hspace{-0.7mm} b \hspace{-0.7mm}
=\hspace{-0.7mm} ( \epsilon\hspace{-0.7mm} \circ\hspace{-0.7mm}
m_{n\tau}) \hspace{-0.7mm} \circ\hspace{-0.7mm}
(n_{\tau}\hspace{-0.7mm} \circ\hspace{-0.7mm} \varepsilon).
\end{eqnarray*}
Proposition \ref{elliptic is transitive} together with Corollary
\ref{simple rigidity} gives the ellipticity of $a$ and $b$.
\end{proof}

Zieve-M\"uller discovered in \cite[Theorem 1.4]{ZM10} a new property
of $(\mathrm{End}(\CC), \circ)$, and we shall prove that the
phenomenon of Zieve-M\"uller remains true in $(\mathrm{End}(\EE),
\circ)$.

\begin{theorem}\label{Zieve-Muller}
Let $\{a, b, f\}\hspace{-0.7mm} \subset\hspace{-0.7mm}
\mathrm{End}(\EE)$, $n\hspace{-0.7mm} =\hspace{-0.7mm}
\deg{f}\hspace{-0.7mm} \ge\hspace{-0.7mm} 2$ and $k\hspace{-0.7mm}
\in\hspace{-0.7mm} \NN$ satisfy $a\hspace{-0.7mm}
\circ\hspace{-0.7mm} b\hspace{-0.7mm} =\hspace{-0.7mm} f^{k}$. If
there exists no  $\iota\hspace{-0.7mm} \in\hspace{-0.7mm}
\mathrm{Aut}(\EE)$ for which  $\iota\hspace{-0.7mm}
\circ\hspace{-0.7mm} f\hspace{-0.7mm} \circ\hspace{-0.7mm}
\iota^{-1}\hspace{-0.7mm} =\hspace{-0.7mm} z^n$  and no
$g\hspace{-0.7mm} \in\hspace{-0.7mm} \mathrm{End}(\EE)$ for which
either $a\hspace{-0.7mm} =\hspace{-0.7mm} f\hspace{-0.7mm}
\circ\hspace{-0.7mm} g$ or $b\hspace{-0.7mm} =\hspace{-0.7mm}
g\hspace{-0.7mm} \circ\hspace{-0.7mm} f$, then $k\hspace{-0.7mm}
\leq\hspace{-0.7mm} \max{\{8, \, 2\mathopen +\mathclose 2
\log_{2}{n}\}}.$
\end{theorem}

The proof of Theorem \ref{Zieve-Muller} relies on techniques
developed in Zieve-M\"uller's original work and therefore our
arguments are largely similar to that in \cite{ZM10}, except the
manipulation of elliptic rational functions. We will be sketchy at
many places.

\begin{lemma}\label{bloom}
Let $a\hspace{-0.7mm} \circ\hspace{-0.7mm} b\hspace{-0.7mm}
=\hspace{-0.7mm} c\hspace{-0.7mm} \circ\hspace{-0.7mm} d$ be a
generalized Ritt relation in $(\mathrm{End}(\EE), \circ)$ with $
b$\,$($resp.\,$a$) neither totally ramified nor elliptic. We have
$\deg{a}\hspace{-0.7mm} <\hspace{-0.7mm} \deg{b}  \ (\,
\mathrm{resp.\,}  \deg{b}\hspace{-0.7mm} <\hspace{-0.7mm} \deg{a}
\,). $
\end{lemma}
\begin{proof}
This follows immediately from Theorem \ref{Ritt move for Blaschke}.
\end{proof}

\begin{lemma}\label{bloom and elliptic}
Given $h\hspace{-0.7mm} \in\hspace{-0.7mm} \mathrm{End}(\EE)$ with
$h(0)\hspace{-0.7mm} \neq\hspace{-0.7mm} 0$ and coprime positive
integers $\{s, n\}$ with $n\hspace{-0.7mm} \ge\hspace{-0.7mm} 2$, if
$z^sh(z)^n$ or $z^sh(z^n)$ is elliptic then we must have
$n\hspace{-0.7mm} =\hspace{-0.7mm} 2$ and $s\hspace{-0.7mm}
=\hspace{-0.7mm} 1$.
\end{lemma}
\begin{proof}
Let $f = z^sh(z)^n$ satisfy the above conditions, then $0 \in
\mathfrak{d}_f.$ There exists $p \in |\mathfrak{D}_f|$ with
$f(\mathfrak{p}) = 0$ and $\mathfrak{D}_f \ge n(\mathfrak{p})$.
Because $f$ is elliptic we have $n = 2$. If $s \ge 2$ then we have
$\mathfrak{D}_f \ge s(0)$, which together with the ellipticity of
$f$ forces $s = 2$. This contradicts to $(s, n)=1.$

Let $f = z^sh(z^n)$ satisfy the above conditions. Take one non-zero
$\mathfrak{p} \in |\mathfrak{D}_f|$ and take a primitive $n$-th root
of unity $\xi_n$, then $\xi_n^i \mathfrak{p} \in |\mathfrak{D}_f|$
for all $0 \leq i \leq n - 1$. By ellipticity $|\mathfrak{D}_f|$ lie
on a geodesic of $\EE$, with respect to the Poincare metric.
Therefore $n = 2$. For the same reason as above we have $s =1$.
\end{proof}



\begin{corollary}\label{all are elliptic}
If $a\mathopen \circ\mathclose b\hspace{-0.7mm}  =\hspace{-0.7mm}
c\mathopen \circ\mathclose d$ is a generalized Ritt relation in
$(\mathrm{End}(\EE), \circ)$ and if $b$ $\mathrm{(}$or
$a$$\mathrm{)}$ is elliptic with degree at least three then
$a\mathopen \circ\mathclose b$ is elliptic.
\end{corollary}
\begin{proof}
Otherwise we are in the first case of Theorem \ref{Ritt move for
Blaschke} with $a\hspace{-0.7mm}=\hspace{-0.7mm}z^n,
b\hspace{-0.7mm}=\hspace{-0.7mm} z^sg(z^n)$\,(\,or
$a\hspace{-0.7mm}=\hspace{-0.7mm}z^sg(z)^n,
b\hspace{-0.7mm}=\hspace{-0.7mm}z^n$). Lemma \ref{bloom and
elliptic} forces $n\hspace{-0.7mm}=\hspace{-0.7mm}2$ and
$s\hspace{-0.7mm}=\hspace{-0.7mm}1$, and this can be checked easily.
\end{proof}

A {\it complete presentation} $\calU$ of $f\hspace{-0.7mm}
\in\hspace{-0.7mm} \mathrm{End}(\EE)\mathopen \setminus\mathclose
\mathrm{Aut}(\EE)$ refers to a tuple $(u_1, \ldots, u_r)$ of
irreducible elements of $(\mathrm{End}(\EE), \circ)$ such that
$f\hspace{-0.7mm}=\hspace{-0.7mm} u_1\hspace{-0.7mm} \circ \cdots
\circ\hspace{-0.7mm} u_r$. If $\calU\hspace{-0.7mm} =\hspace{-0.7mm}
(u_1, \ldots, u_r), \calV\hspace{-0.7mm} =\hspace{-0.7mm} (v_1,
\ldots, v_r)$ are complete presentations of $f$ then by Theorem
\ref{NgWangmonoid} we can pass from $\calU$ to $\calV$ by finitely
many Ritt relations, and this gives a unique permutation
$\sigma_{\calU, \calV}$ of $\{1, 2, \ldots, r\}$ which satisfies
$\deg{u_i}\hspace{-0.7mm} =\hspace{-0.7mm} \deg{v_{\sigma_{\calU,
\calV}(i)}}$. In addition we have
\begin{lemma}\label{preserve}
If $i\hspace{-0.7mm} <\hspace{-0.7mm} j$ and if $\sigma_{\calU,
\calV}(i)\hspace{-0.7mm} > \hspace{-0.7mm} \sigma_{\calU, \calV}(j)$
then $(\deg{u_i}, \deg{u_j})\hspace{-0.7mm} =\hspace{-0.7mm} 1.$
\end{lemma}
Following \cite{ZM10} we define $ LL(\calU, \calV, i, j) = \prod_{k
< i, \sigma(k) < \sigma(j)} \deg{u_k},  LR(\calU, \calV, i, j)=
\prod_{k < i, \sigma(k) > \sigma(j)} \deg{u_k},  RL(\calU, \calV, i,
j) = \prod_{k > i, \sigma(k) < \sigma(j)} \deg{u_k}$ and  $RR(\calU,
\calV, i, j)= \prod_{k
> i, \sigma(k) > \sigma(j)} \deg{u_k}. $ Let $\calU\hspace{-0.7mm} =\hspace{-0.7mm} (u_1, \ldots, u_r)$ be a
complete presentation of a $f\hspace{-0.7mm} \in\hspace{-0.7mm}
\mathrm{End}(\EE)\mathopen \setminus\mathclose \mathrm{Aut}(\EE)$
and let $u_k\hspace{-0.7mm} \in\hspace{-0.7mm} \calU$ be elliptic
with $\deg{u_k}\hspace{-0.7mm} \ge\hspace{-0.7mm} 3$. The {\it
length} of $u_k$ with respect to $\calU$, denoted by
$h_{\calU}(u_k)$ or $h(u_k)$ if without ambiguity, is defined as
$h_{\calU}(u_k)\hspace{-0.7mm} =\hspace{-0.7mm} \Pi_{i\mathopen
=\mathclose 1}^{k\mathopen-\mathclose 1} \deg{u_i}$. If $u_i$ is
elliptic, then according to Theorem \ref{NgWangmonoid} so is
$v_{\sigma_{\calU, \calV}(i)}$. Indeed
\begin{lemma}\label{heigh of chebyshev}
If $u_i$ is elliptic with degree at least three then
\begin{eqnarray*}
h(u_i)\chi(u_i) = h(v_{\sigma_{\calU, \calV}(i)})
\chi(v_{\sigma_{\calU, \calV}(i)}).
\end{eqnarray*}
\end{lemma}
\begin{proof}
This follows from Corollary \ref{all are elliptic}, Proposition
\ref{nesting general} and Corollary \ref{simple rigidity}.
\end{proof}

Moreover we also have
\begin{lemma}\label{heigh of chebyshev2}
If $i\hspace{-0.7mm} <\hspace{-0.7mm} j$ and if \{$u_i$, $u_j$,
$u_i\hspace{-0.7mm} \circ\hspace{-0.7mm} u_{i\mathopen+\mathclose
1}\hspace{-0.7mm} \circ \cdots \circ\hspace{-0.7mm}
u_{j}\hspace{-0.7mm}\sim\hspace{-0.7mm}\mathcal{T}_{n,\, t}\}
\hspace{-0.7mm} \subset \hspace{-0.7mm} \mathrm{End}(\EE)$ are all
elliptic and of degree at least three then
\begin{eqnarray*}
h(u_i)\chi(u_i)\hspace{-0.7mm} =\hspace{-0.7mm} h(u_j)\chi(u_j).
\end{eqnarray*}
\end{lemma}
\begin{proof}
Writing $\deg{u_i}\hspace{-0.7mm}=\hspace{-0.7mm}d_i$, by Corollary
\ref{simple rigidity} we have
$u_{j}\hspace{-0.7mm}$\,(resp.\,$u_{i}$) is associated to
$\mathcal{T}_{d_j, t}$\,(resp.\,$\mathcal{T}_{d_i, r}$ with $r
\hspace{-0.7mm}=\hspace{-0.7mm}t\Pi_{k=i+1}^{j}d_k$). and therefore
$\chi(u_j)\hspace{-0.7mm} =\hspace{-0.7mm} d_jt$ and
$\chi(u_i)\hspace{-0.7mm} =\hspace{-0.7mm} d_i r$. It is also clear
that
$h(u_j)\hspace{-0.7mm}=\hspace{-0.7mm}\Pi_{k=i}^{j\mathopen-\mathclose
1}d_k$ and $h(u_i)\hspace{-0.7mm}=\hspace{-0.7mm}1$. The claim
follows readily.
\end{proof}
\begin{proposition}\label{complicated}
Let $f\hspace{-0.7mm} \in\hspace{-0.7mm} \mathrm{End}(\EE)\mathopen
\setminus\mathclose \mathrm{Aut}(\EE)$, $\calU\hspace{-0.7mm}
=\hspace{-0.7mm} (u_1, \ldots, u_r)$ and $\calV\hspace{-0.7mm}
=\hspace{-0.7mm} (v_1, \ldots, v_r)$ its complete presentations and
$1\hspace{-0.7mm}\leq\hspace{-0.7mm}k\hspace{-0.7mm}\leq\hspace{-0.7mm}r$.
Writing $LL\hspace{-0.7mm} =\hspace{-0.7mm} LL(\calU, \calV, k, k)$
and $LR, RL, RR$ analogously,  then $LR, RL$ are both coprime to
$\deg{u_k}$ and there exist finite Blaschke products $a$ with degree
$LL$, $d$ with degree $RR$, $b, \hat{b}, \tilde{b}$ with degree
$LR$, $c, \tilde{c}, \bar{c}$ with degree $RL$ and $\hat{u},
\tilde{u}, \bar{u}$ with degree $\deg{u_k}$
such that
\begin{enumerate}
\item[(i)] $u_1\hspace{-0.7mm} \circ\hspace{-0.7mm}
u_2\hspace{-0.7mm} \circ \cdots \circ\hspace{-0.7mm}
u_{k\mathopen-\mathclose 1}\hspace{-0.7mm} =\hspace{-0.7mm}
a\hspace{-0.7mm} \circ\hspace{-0.7mm} b$ and
$u_{k\mathopen+\mathclose 1}\hspace{-0.7mm} \circ \cdots
\circ\hspace{-0.7mm} u_{r}\hspace{-0.7mm} =\hspace{-0.7mm}
c\hspace{-0.7mm} \circ\hspace{-0.7mm} d;$
\item[(ii)] $b\hspace{-0.7mm} \circ\hspace{-0.7mm} u_k\hspace{-0.7mm}
=\hspace{-0.7mm} \hat{u}\hspace{-0.7mm} \circ\hspace{-0.7mm}
\hat{b};$
\item[(iii)] $\hat{u}\hspace{-0.7mm} \circ\hspace{-0.7mm}
\hat{b}\hspace{-0.7mm} \circ\hspace{-0.7mm} c\hspace{-0.7mm}
=\hspace{-0.7mm} \tilde{c}\hspace{-0.7mm} \circ\hspace{-0.7mm}
\tilde{u}\hspace{-0.7mm} \circ\hspace{-0.7mm} \tilde{b};$
\item[(iv)] $u_k\hspace{-0.7mm} \circ\hspace{-0.7mm} c\hspace{-0.7mm}
=\hspace{-0.7mm} \bar{c}\hspace{-0.7mm} \circ\hspace{-0.7mm}
\bar{u}.$
\end{enumerate}
\end{proposition}
\begin{proof}
Based on Proposition \ref{rigidity of Blaschke}, some analysis
similar to that in proof of \cite[Proposition 4.2]{ZM10} applies to
our case.
\end{proof}

\begin{proofof}{Proof of Theorem \ref{Zieve-Muller}} We assume
that $k\hspace{-0.7mm} \ge\hspace{-0.7mm} 2$. Choose
$\calU\hspace{-0.7mm} =\hspace{-0.7mm} (u_1, \ldots, u_r)$ to be a
complete presentation of $f$,  then $\calU^k\hspace{-0.7mm}
=\hspace{-0.7mm} (u_1, \ldots, u_{kr})$ is a complete presentation
of $f^{k}$ where $u_i\hspace{-0.7mm} =\hspace{-0.7mm}
u_{i\mathopen-\mathclose r}$. Let $\calV\hspace{-0.7mm}
=\hspace{-0.7mm} (v_1, \ldots, v_{kr})$ be a complete presentation
of $f^{k}$ for which $a\hspace{-0.7mm} =\hspace{-0.7mm}
v_1\hspace{-0.7mm} \circ\hspace{-0.7mm} v_2 \cdots
\circ\hspace{-0.7mm} v_{e}$ and $b\hspace{-0.7mm} =\hspace{-0.7mm}
v_{e\mathopen+\mathclose 1}\hspace{-0.7mm} \circ \cdots
\circ\hspace{-0.7mm} v_{kr}$. By the assumption that
$f^{k}\hspace{-0.7mm} =\hspace{-0.7mm} a\hspace{-0.7mm}
\circ\hspace{-0.7mm} b$ and that there does exist no
$g\hspace{-0.7mm} \in\hspace{-0.7mm} \mathrm{End}(\EE)$ for which
$a\hspace{-0.7mm} =\hspace{-0.7mm} f\hspace{-0.7mm}
\circ\hspace{-0.7mm} g$ or $b\hspace{-0.7mm} =\hspace{-0.7mm}
g\hspace{-0.7mm} \circ\hspace{-0.7mm} f$, Proposition \ref{rigidity
of Blaschke} applies and leads to $\deg{f}\hspace{-0.7mm}
\nmid\hspace{-0.7mm} \deg{a}$ and $\deg{f}\hspace{-0.7mm}
\nmid\hspace{-0.7mm} \deg{b}$. Therefore there exists
$1\hspace{-0.7mm} \leq\hspace{-0.7mm} m\hspace{-0.7mm}
\leq\hspace{-0.7mm} r$\,(resp.\,$1\hspace{-0.7mm}
\leq\hspace{-0.7mm} l\hspace{-0.7mm} \leq\hspace{-0.7mm} r$) such
that $\sigma_{\calU^k,\, \calV}(m\mathopen +\mathclose
tr)\hspace{-0.7mm}
>\hspace{-0.7mm} e$ for all $0\hspace{-0.7mm} \leq\hspace{-0.7mm} t\hspace{-0.7mm} \leq\hspace{-0.7mm} k\mathopen -\mathclose
1$\,(resp.\,$\sigma_{\calU^{k},\, \calV}(l\mathopen +\mathclose tr)
\hspace{-0.7mm} \leq\hspace{-0.7mm}  e$ for all $0\hspace{-0.7mm}
\leq\hspace{-0.7mm} t\hspace{-0.7mm} \leq\hspace{-0.7mm} k\mathopen
-\mathclose 1$). Otherwise Proposition \ref{rigidity of Blaschke}
leads to a contradiction. Moreover by Lemma \ref{preserve} we have
$(\deg{u_m}, \deg{u_l})\hspace{-0.7mm} =\hspace{-0.7mm} 1.$

\noindent {\it Case {\it (i)}}, there exists $1\hspace{-0.7mm}
\leq\hspace{-0.7mm} p\hspace{-0.7mm} \leq\hspace{-0.7mm} r$ such
that $u_p$ is not associated to $z^n$, $\mathcal{T}_{n,\,t}$, $z^s
h(z^n)$ or $z^sh(z^n)$ in $(\mathrm{End}(\EE), \circ)$ with
$h\hspace{-0.7mm} \in\hspace{-0.7mm} \mathrm{End}(\EE)$,
$h(0)\hspace{-0.7mm} \neq\hspace{-0.7mm}0$ and $n\hspace{-0.7mm}
\ge\hspace{-0.7mm} 2$.

We claim that $k\hspace{-0.7mm} =\hspace{-0.7mm} 2$. Otherwise we
have $k\hspace{-0.7mm} \ge\hspace{-0.7mm} 3$. On the one hand we
deduce from Theorem \ref{Ritt move for Blaschke} that $u_{p + r}$
never changes under Ritt relations and therefore $\sigma_{\calU^k,\,
\calV}(i)\hspace{-0.7mm} <\hspace{-0.7mm} p\mathopen +\mathclose r$
for all $i\hspace{-0.7mm} <\hspace{-0.7mm} p\mathopen +\mathclose r$
and $\sigma_{\calU^k,\, \calV}(i) \hspace{-0.7mm}>\hspace{-0.7mm}
p\mathopen +\mathclose r$ for all $i\hspace{-0.7mm}
>\hspace{-0.7mm} p\mathopen +\mathclose r$, which leads to $\sigma_{\calU^k,\,
\calV}(m)\hspace{-0.7mm} <\hspace{-0.7mm} p\mathopen +\mathclose r$
and $\sigma_{\calU^k,\, \calV}(l\mathopen +\mathclose(k\mathopen
-\mathclose 1)r) \hspace{-0.7mm}>\hspace{-0.7mm} p\mathopen
+\mathclose r$. On the other hand we have $\sigma_{\calU^k,\,
\calV}(m) \hspace{-0.7mm}>\hspace{-0.7mm} e$ and $\sigma_{\calU^k,\,
\calV}(l\mathopen +\mathclose (k\mathopen -
\mathclose1)r)\hspace{-0.7mm} \leq\hspace{-0.7mm} e$. Consequently
$e\hspace{-0.7mm} <\hspace{-0.7mm} p\mathopen +\mathclose r$ and
$p\mathopen +\mathclose r\hspace{-0.7mm} <\hspace{-0.7mm} e$, a
contradiction.

\noindent {\it Case {\it (ii)}}, there exists  $1\hspace{-0.7mm}
\leq\hspace{-0.7mm} p\hspace{-0.7mm} \leq\hspace{-0.7mm} r$ such
that $u_p$ is neither totally ramified nor elliptic, but associated
to $z^s h(z^n)$ or $z^sh(z)^n$ in $(\mathrm{End}(\EE), \circ)$ with
$h\hspace{-0.7mm} \in\hspace{-0.7mm} \mathrm{End}(\EE)$,
$h(0)\hspace{-0.7mm} \neq\hspace{-0.7mm}0$ and $n\hspace{-0.7mm}
\ge\hspace{-0.7mm} 2$.

There exists $0\hspace{-0.7mm} \leq\hspace{-0.7mm} q\hspace{-0.7mm}
\leq\hspace{-0.7mm} k\mathopen -\mathclose 1$ for which
$\sigma_{\calU^k, \calV}(p\mathopen+\mathclose qr)\hspace{-0.7mm}
\leq\hspace{-0.7mm} e$ and $\sigma_{\calU^k,
\calV}(p\mathopen+\mathclose
(q\mathopen+\mathclose1)r)\hspace{-0.7mm}
>\hspace{-0.7mm} e$. Because $\sigma_{\calU^k, \calV}(m\mathopen+\mathclose
tr)\hspace{-0.7mm}
>\hspace{-0.7mm} e$ for all $0\hspace{-0.7mm} \leq\hspace{-0.7mm} t\hspace{-0.7mm} \leq\hspace{-0.7mm} q\mathopen -\mathclose 1$
Proposition \ref{complicated} gives
\begin{eqnarray*}
(\deg{u_m})^{q} | LR(p\mathopen+\mathclose qr).
\end{eqnarray*}
Similarly, because $\sigma_{\calU^k, \calV}(u_{l\mathopen
+\mathclose tr})\hspace{-0.7mm} \leq\hspace{-0.7mm} e$ for all
$q\mathopen +\mathclose 2\hspace{-0.7mm} \leq\hspace{-0.7mm}
t\hspace{-0.7mm} \leq\hspace{-0.7mm} k\mathopen -\mathclose 1$ we
have
\begin{eqnarray*}
(\deg{u_l})^{k-q-2}|RL(p\mathopen+\mathclose
(q\mathopen+\mathclose1)r).
\end{eqnarray*}
By Corollary \ref{bloom} and by Proposition \ref{complicated} we
have
\begin{eqnarray*}
(\deg{u_m})^{q}\hspace{-0.7mm} <\hspace{-0.7mm} \deg{u_p} ,  \ \ \
(\deg{u_l})^{k-q-2}\hspace{-0.7mm} <\hspace{-0.7mm} \deg{u_p}.
\end{eqnarray*}
This gives $2^{k\mathopen-\mathclose 2}\hspace{-0.7mm}
\leq\hspace{-0.7mm} (\deg{u_p})^2\hspace{-0.7mm} \leq\hspace{-0.7mm}
n^2$ and therefore $k\hspace{-0.7mm} \leq\hspace{-0.7mm} 2\mathopen
+\mathclose 2 \log_2 n$ as desired.

\noindent {\it Case {\it (iii)}}, all
$u_i\mathpunct{:}\EE\hspace{-0.6mm}\to \hspace{-0.6mm} \EE$ in
$\calU$ are totally ramified.

If $|\mathfrak{D}_{u_i}|\hspace{-0.7mm} =\hspace{-0.7mm}
\mathfrak{d}_{u_{i\mathopen +\mathclose 1}}\hspace{-0.7mm}
=\hspace{-0.7mm}\mathfrak{p}$ holds for all integer $i$ with
$1\hspace{-0.7mm} \leq\hspace{-0.7mm} i\hspace{-0.7mm}
\leq\hspace{-0.7mm} kr\mathopen -\mathclose 1$ then
$\iota_{\mathfrak{p}}\hspace{-0.7mm} \circ\hspace{-0.7mm}
f\hspace{-0.7mm} \circ\hspace{-0.7mm}
\iota_{-\mathfrak{p}}\hspace{-0.7mm} =\hspace{-0.7mm} \zeta z^n$ for
some $\zeta \hspace{-0.7mm} \in\hspace{-0.7mm} \TT$, which
contradicts to the assumption. Hence there exists $1\hspace{-0.7mm}
\leq\hspace{-0.7mm} p\hspace{-0.7mm} \leq\hspace{-0.7mm} r$ such
that $|\mathfrak{D}_{u_p}|\hspace{-0.7mm} \neq\hspace{-0.7mm}
\mathfrak{d}_{u_{p\mathopen +\mathclose 1}}$. It is clear from
Theorem \ref{Ritt move for Blaschke} and from Corollary \ref{bloom
and totally ramified} that any Ritt relation $a\hspace{-0.7mm}
\circ\hspace{-0.7mm} b\hspace{-0.7mm} =\hspace{-0.7mm}
c\hspace{-0.7mm} \circ\hspace{-0.7mm} d$ in totally ramified finite
Blaschke products  must satisfy $\mathfrak{d}_a\hspace{-0.7mm}
=\hspace{-0.7mm} \mathfrak{d}_c$ and
$|\mathfrak{D}_b|\hspace{-0.7mm} =\hspace{-0.7mm} |\mathfrak{D}_d|$.
This implies that if $i\hspace{-0.7mm} \leq\hspace{-0.7mm}
r\mathopen +\mathclose p$\,(resp.\,$(k\mathopen -\mathclose
2)r\mathopen +\mathclose p+\!1\hspace{-0.7mm} \leq\hspace{-0.7mm}
i$) then $\sigma_{\calU^k,\calV}(i)\hspace{-0.7mm}
\leq\hspace{-0.7mm} r\mathopen +\mathclose
p$\,(resp.\,$\sigma_{\calU^k,\calV}(i)\hspace{-0.7mm}
\ge\hspace{-0.7mm} (k\mathopen -\mathclose 2)r\mathopen +\mathclose
p+\!1$), which leads to $\sigma_{\calU^k,\calV}(m)\hspace{-0.7mm}
\leq\hspace{-0.7mm} r\mathopen +\mathclose p$ and
$\sigma_{\calU^k,\calV}((k\mathopen -\mathclose 1)r\mathopen
+\mathclose l)\hspace{-0.7mm} \ge\hspace{-0.7mm} (k\mathopen
-\mathclose 2)r\mathopen +\mathclose p+\!1$. Using
$\sigma_{\calU^k,\, \calV}(m)\hspace{-0.7mm}
>\hspace{-0.7mm} e$ and $\sigma_{\calU^k,\, \calV}(l\mathopen +\mathclose (k\mathopen -\mathclose 1)r)\hspace{-0.7mm} \leq\hspace{-0.7mm}
e$, we obtain $e\hspace{-0.7mm} <\hspace{-0.7mm} p\mathopen
+\mathclose r$ and $(k\mathopen -\mathclose 2)r\mathopen +\mathclose
p + 1\hspace{-0.7mm} \leq\hspace{-0.7mm} e$ which forces
$k\hspace{-0.7mm} \leq\hspace{-0.7mm} 2$.

\noindent {\it Case {\it (iv)}},  there exist $1\hspace{-0.7mm}
\leq\hspace{-0.7mm} p\hspace{-0.7mm} \leq\hspace{-0.7mm} r$ such
that $u_p$ is elliptic and is of  degree at least three.

We claim that $k\hspace{-0.7mm} \leq\hspace{-0.7mm} 8$. Otherwise
$k\hspace{-0.7mm} \ge\hspace{-0.7mm} 9$ and either
$\sigma_{\calU^k,\calV}(4r\mathopen+\mathclose p)\hspace{-0.7mm}
\leq\hspace{-0.7mm} e$ or $\sigma_{\calU^k,\calV}(4r\mathopen
+\mathclose p)\hspace{-0.7mm}
>\hspace{-0.7mm} e$. In the former case we deduce from Proposition \ref{complicated} that
there exist $\{a, b, \hat{u}, \overline{b}\} \hspace{-0.7mm}
\subset\hspace{-0.7mm} \mathrm{End}(\EE)$ such that
$\deg{b}\hspace{-0.7mm} =\hspace{-0.7mm}
\deg{\overline{b}}\hspace{-0.7mm} =\hspace{-0.7mm} LR(\calU^k,
\calV,p\mathopen+\mathclose4r,
p\mathopen+\mathclose4r)\hspace{-0.7mm} =\hspace{-0.7mm} \hat{n},
\deg{\hat{u}}\hspace{-0.7mm} =\hspace{-0.7mm} \deg{u}$ and \bgroup
\arraycolsep=2pt
\begin{eqnarray*}
u_1\hspace{-0.7mm} \circ\hspace{-0.7mm} u_2\hspace{-0.7mm} \circ \cdots \circ\hspace{-0.7mm} u_{p\mathopen+\mathclose4r\mathopen-\mathclose1}\hspace{-0.7mm} &=&\hspace{-0.7mm} a\hspace{-0.7mm} \circ\hspace{-0.7mm} b, \\
b\hspace{-0.7mm} \circ\hspace{-0.7mm} u_{p\mathopen+\mathclose
4r}\hspace{-0.7mm} &=&\hspace{-0.7mm} \hat{u}\hspace{-0.7mm}
\circ\hspace{-0.7mm} \overline{b}.
\end{eqnarray*}
\egroup Because $\sigma_{\calU^k,\calV}(r\mathopen+\mathclose
m)\hspace{-0.7mm}
>\hspace{-0.7mm} \sigma_{\calU^k,\calV}(m) \hspace{-0.7mm}
>\hspace{-0.7mm} e\hspace{-0.7mm} \ge\hspace{-0.7mm} \sigma_{\calU^k,\calV}(4r\mathopen+\mathclose p)$ we have $\deg{b}\hspace{-0.7mm} =\hspace{-0.7mm} \hat{n}\hspace{-0.7mm} \ge\hspace{-0.7mm}
4$ and because $b\hspace{-0.7mm} \circ\hspace{-0.7mm}
u_{p\mathopen+\mathclose 4r}\hspace{-0.7mm} =\hspace{-0.7mm}
\hat{u}\hspace{-0.7mm} \circ\hspace{-0.7mm} \overline{b}$ is a
generalized Ritt relation, Corollary \ref{all are elliptic} implies
that $b$ is elliptic. If we write $\bar{n}\hspace{-0.7mm}
=\hspace{-0.7mm} LR(\calU^k, \calV, p\mathopen+\mathclose 2r,
p\mathopen+\mathclose 4r)$, $h\hspace{-0.7mm} =\hspace{-0.7mm}
u_1\hspace{-0.7mm} \circ\hspace{-0.7mm} u_2\hspace{-0.7mm} \circ
\cdots \circ\hspace{-0.7mm}
u_{p\mathopen+\mathclose2r\mathopen-\mathclose1}$ and
$g\hspace{-0.7mm} =\hspace{-0.7mm}
u_{p\mathopen+\mathclose2r}\hspace{-0.7mm} \circ\hspace{-0.7mm}
u_{p\mathopen+\mathclose2r\mathopen+\mathclose1}\hspace{-0.7mm}
\circ \cdots \circ\hspace{-0.7mm}
u_{p\mathopen+\mathclose4r\mathopen-\mathclose1}$ then apparently
\begin{eqnarray*}
a\hspace{-0.7mm} \circ\hspace{-0.7mm} b\hspace{-0.7mm}
=\hspace{-0.7mm} h\hspace{-0.7mm} \circ\hspace{-0.7mm} g,
\end{eqnarray*}
and for the same reason to that for $\hat{n}$ we have
$\bar{n}\hspace{-0.7mm} \ge\hspace{-0.7mm} 4$. By Proposition
\ref{rigidity of Blaschke} there exist $\{\hat{b}, \hat{g}, k, e,
\hat{a}, \hat{h}\}\hspace{-0.7mm}
\subset\hspace{-0.7mm}\mathrm{End}(\EE)$ with
$\deg{k}\hspace{-0.7mm} =\hspace{-0.7mm} (\hat{n}\hspace{-0.7mm}
=\hspace{-0.7mm} \deg{b}, \deg{g}), \deg{e}\hspace{-0.7mm}
=\hspace{-0.7mm} (\deg{a}, \deg{h})$ and \bgroup \arraycolsep=2pt
\begin{eqnarray*}
 \hat{b}\hspace{-0.7mm} \circ\hspace{-0.7mm} k\hspace{-0.7mm} &=&\hspace{-0.7mm} b, \ \ \ \hat{g}\hspace{-0.7mm} \circ\hspace{-0.7mm} k\hspace{-0.7mm} =\hspace{-0.7mm} g, \\
e\hspace{-0.7mm} \circ\hspace{-0.7mm} \hat{a}\hspace{-0.7mm} &=&\hspace{-0.7mm} a, \ \ \ e\hspace{-0.7mm} \circ\hspace{-0.7mm} \hat{h}\hspace{-0.7mm} =\hspace{-0.7mm} h ,\\
\hat{a}\hspace{-0.7mm} \circ\hspace{-0.7mm} \hat{b}\hspace{-0.7mm}
&=&\hspace{-0.7mm} \hat{h}\hspace{-0.7mm} \circ\hspace{-0.7mm}
\hat{g}.
\end{eqnarray*}
\egroup We denote $\deg{k}$ by $s$ and consider the generalized Ritt
relation $\hat{h}\hspace{-0.7mm} \circ\hspace{-0.7mm}
\hat{g}\hspace{-0.7mm} =\hspace{-0.7mm} \hat{a}\hspace{-0.7mm}
\circ\hspace{-0.7mm} \hat{b}$. Because
\begin{eqnarray*}
\hat{n}/\bar{n} = \prod_{p\mathopen+\mathclose 2r \leq i \leq
p\mathopen+\mathclose 4r\mathopen-\mathclose1, \,
\sigma_{\calU^k,\calV}(i)
> \sigma_{\calU^k,\calV}(p\mathopen+\mathclose 4r)} \deg{u_i}
\end{eqnarray*}
and because for all $p\mathopen +\mathclose 2r\hspace{-0.7mm}
\leq\hspace{-0.7mm} i\hspace{-0.7mm} \leq\hspace{-0.7mm} p\mathopen
+\mathclose 4r\mathopen -\mathclose 1$
\begin{eqnarray*}
(\deg{u_i}, \bar{n})\hspace{-0.7mm} >\hspace{-0.7mm} 1 \Rightarrow
\sigma_{\calU^k,\calV}(i)\hspace{-0.7mm} >\hspace{-0.7mm}
\sigma_{\calU^k,\calV}(p\mathopen+\mathclose 4r),
\end{eqnarray*}
we have $ (\deg{g}/(\hat{n}/\bar{n}), \bar{n})\hspace{-0.7mm}
=\hspace{-0.7mm} 1, $ and therefore
$s\hspace{-0.7mm}=\hspace{-0.7mm}(\hat{n},
\deg{g})\hspace{-0.7mm}=\hspace{-0.7mm}\hat{n}/\bar{n}$ or
equivalently
\begin{eqnarray*}
\deg{k} = \prod_{p+2r \leq i \leq p+4r-1, \,
\sigma_{\calU^k,\calV}(i) > \sigma_{\calU^k,\calV}(p+4r)} \deg{u_i}.
\end{eqnarray*}
Because
$\sigma_{\calU^k,\calV}(p\mathopen+\mathclose2r)\hspace{-0.7mm}
<\hspace{-0.7mm}\sigma_{\calU^k,\calV}(p\mathopen+\mathclose3r)\hspace{-0.7mm}
<\hspace{-0.7mm} \sigma_{\calU^k,\calV}(p\mathopen+\mathclose4r)$
the above equality leads to $\deg{\hat{g}}\hspace{-0.7mm}
\ge\hspace{-0.7mm} 4$. The ellipticity  of $b$ implies that of
$\hat{b}$. Noticing that $\deg{\hat{b}}\hspace{-0.7mm} =
\hspace{-0.7mm} \deg{b}/\deg{k} \hspace{-0.7mm} = \hspace{-0.7mm}
\hat{n}/\deg{k}$ and $\deg{k}\hspace{-0.7mm}=\hspace{-0.7mm}
\hat{n}/\bar{n}$, we have
$\deg{\hat{b}}\hspace{-0.7mm}=\hspace{-0.7mm} \bar{n}\hspace{-0.7mm}
\ge \hspace{-0.7mm} 4$.   Considering the generalized Ritt relation
$\hat{h}\hspace{-0.7mm} \circ\hspace{-0.7mm} \hat{g}\hspace{-0.7mm}
=\hspace{-0.7mm} \hat{a}\hspace{-0.7mm} \circ\hspace{-0.7mm}
\hat{b}$  Corollary \ref{all are elliptic} implies the ellipticity
of $\hat{g}$. We now examine
\begin{eqnarray*}
g\hspace{-0.7mm} =\hspace{-0.7mm} u_{p\mathopen+\mathclose
2r}\hspace{-0.7mm} \circ\hspace{-0.7mm} u_{p\mathopen+\mathclose
2r\mathopen+\mathclose 1}\hspace{-0.7mm} \circ \cdots
\circ\hspace{-0.7mm} u_{p\mathopen+\mathclose 4r\mathopen-\mathclose
1}\hspace{-0.7mm} =\hspace{-0.7mm} \hat{g}\hspace{-0.7mm}
\circ\hspace{-0.7mm} k
\end{eqnarray*}
and we write $\overline{\calU}\hspace{-0.7mm} =\hspace{-0.7mm}
(\overline{u}_{1}\hspace{-0.7mm} =\hspace{-0.7mm}
u_{p\mathopen+\mathclose 2r}, \ldots,
\overline{u}_{2r}\hspace{-0.7mm} =\hspace{-0.7mm}
u_{p\mathopen+\mathclose 4r\mathopen-\mathclose 1})$ which is a
complete presentation of $g$. If $\overline{\calV}\hspace{-0.7mm}
=\hspace{-0.7mm} (v_1, \ldots, v_{2r})$ is a complete presentation
of $g$ for which $\hat{g}\hspace{-0.7mm} =\hspace{-0.7mm}
v_1\hspace{-0.7mm} \circ\hspace{-0.7mm} v_2 \cdots
\circ\hspace{-0.7mm} v_{o}$ and $k\hspace{-0.7mm} =\hspace{-0.7mm}
v_{o\mathopen+\mathclose 1}\hspace{-0.7mm} \circ \cdots
\circ\hspace{-0.7mm} v_{2r}$ then $\sigma_{\overline{\calU},
\overline{\calV}}(1)\hspace{-0.7mm} \leq\hspace{-0.7mm} o$ and
$\sigma_{\overline{\calU}, \overline{\calV}}(1\mathopen+\mathclose
r)\hspace{-0.7mm} \leq\hspace{-0.7mm} o$. Lemma \ref{heigh of
chebyshev2} gives
\begin{eqnarray*}
h(v_{\sigma_{\overline{\calU},
\overline{\calV}}(1)})\chi(v_{\sigma_{\overline{\calU},
\overline{\calV}}(1)}) = h(v_{\sigma_{\overline{\calU},
\overline{\calV}}(1+r)})\chi(v_{\sigma_{\overline{\calU},
\overline{\calV}}(1+r)}),
\end{eqnarray*}
then we apply Lemma \ref{heigh of chebyshev} and have
\begin{eqnarray*}
h(\overline{u}_{1})\chi(\overline{u}_{1}) =
h(\overline{u}_{1+r})\chi(\overline{u}_{1+r}).
\end{eqnarray*}
This is impossible since $\chi(\overline{u}_{1})\hspace{-0.7mm}
=\hspace{-0.7mm} \chi(\overline{u}_{1\mathopen+\mathclose
r})\hspace{-0.7mm} =\hspace{-0.7mm} \chi(u_p)$ and
$h(\overline{u}_{1})\hspace{-0.7mm} <\hspace{-0.7mm}
h(\overline{u}_{r\mathopen+\mathclose 1})$.

Similar arguments apply to the case that
$\sigma_{\calU^k,\calV}(4r\mathopen+\mathclose p)\hspace{-0.7mm}
>\hspace{-0.7mm} e$.
\end{proofof}

\begin{corollary}\label{Zieve-Muller, corallary}
Let $f\hspace{-0.7mm} \in\hspace{-0.7mm} \mathrm{End}(\EE)\mathopen
\setminus\mathclose \mathrm{Aut}(\EE), \{a, b\} \hspace{-0.7mm}
\subset \hspace{-0.7mm} \mathrm{End}(\EE)$  and $l\hspace{-0.7mm}
\ge \hspace{-0.7mm} 1$ that satisfy $a\hspace{-0.7mm}
\circ\hspace{-0.7mm} b\hspace{-0.7mm} =\hspace{-0.7mm} f^{l}$ and
there exist no $\iota\hspace{-0.7mm} \in\hspace{-0.7mm}
\mathrm{Aut}{(\EE)}$ for which $\iota\hspace{-0.7mm}
\circ\hspace{-0.7mm} f\hspace{-0.7mm} \circ\hspace{-0.7mm}
\iota\!^{-\!1}\hspace{-0.7mm}=\hspace{-0.7mm}z^{\deg{f}}$. Then
there exist $\{\overline{a\bPh}, \overline{b}\}\hspace{-0.7mm}
\subset\hspace{-0.7mm} \mathrm{End}(\EE)$ and nonnegative integers
$k\hspace{-0.7mm} \leq\hspace{-0.7mm} \max{(8, 2\mathopen
+\mathclose 2\log_{2}{\deg{f}})}, i, j$ such that
\begin{eqnarray*}
a\hspace{-0.7mm} =\hspace{-0.7mm} f^{i}\hspace{-0.7mm}
\circ\hspace{-0.7mm} \overline{a\bPh},  \ \ \ \ \ \ b\hspace{-0.7mm}
=\hspace{-0.7mm} \overline{b}\hspace{-0.7mm} \circ\hspace{-0.7mm}
f^{j}, \ \ \
 \ \ \ \overline{a\bPh}\hspace{-0.7mm} \circ\hspace{-0.7mm} \overline{b}\hspace{-0.7mm} =\hspace{-0.7mm} f^{k}.
\end{eqnarray*}
\end{corollary}
\begin{proof}
Let $i$\,(resp.\,$j$) be the maximal nonnegative integer that
$a\hspace{-0.7mm} =\hspace{-0.7mm} f^{i}\hspace{-0.7mm}
\circ\hspace{-0.7mm} \overline{a\bPh}$\,(resp.\,$b\hspace{-0.7mm}
=\hspace{-0.7mm} \overline{b}\hspace{-0.7mm} \circ\hspace{-0.7mm}
f^{j}$) for some $\overline{a\bPh}$\,(resp.\,$\overline{b}$) in
$\mathrm{End}(\EE)$. We have $f^{i}\hspace{-0.7mm}
\circ\hspace{-0.7mm} \overline{a\bPh} \hspace{-0.7mm} \circ
\hspace{-0.7mm} \overline{b}\hspace{-0.7mm} \circ\hspace{-0.7mm}
f^{j} \hspace{-0.7mm} = \hspace{-0.7mm} f^l$ and therefore
$f^{i}\hspace{-0.7mm} \circ\hspace{-0.7mm} \overline{a\bPh}
\hspace{-0.7mm} \circ \hspace{-0.7mm} \overline{b}\hspace{-0.7mm} =
\hspace{-0.7mm} f^{l\mathopen-\mathclose j}$. This together with
Corollary \ref{simple rigidity} implies that there exists
$\epsilon\hspace{-0.7mm} \in\hspace{-0.7mm} \mathrm{Aut}(\EE)$ for
which
\begin{eqnarray*}
f^i \hspace{-0.7mm} = \hspace{-0.7mm} f^i \hspace{-0.7mm} \circ
\hspace{-0.7mm} \epsilon^{-1}, \ \ \ \overline{a\bPh}
\hspace{-0.7mm} \circ \hspace{-0.7mm} \overline{b}\hspace{-0.7mm} =
\hspace{-0.7mm} \epsilon \hspace{-0.7mm} \circ \hspace{-0.7mm}
f^{l\mathopen -\mathclose i\mathopen -\mathclose j}.
\end{eqnarray*}
Replacing $\overline{a}$ by $\epsilon^{-1}\hspace{-0.7mm} \circ
\hspace{-0.3mm} \overline{a} $ we have $a\hspace{-0.7mm}
=\hspace{-0.7mm} f^{i}\hspace{-0.7mm} \circ\hspace{-0.7mm}
\overline{a\bPh},  b\hspace{-0.7mm} =\hspace{-0.7mm}
\overline{b}\hspace{-0.7mm} \circ\hspace{-0.7mm} f^{j}$ and
$\overline{a\bPh}\hspace{-0.7mm} \circ\hspace{-0.7mm}
\overline{b}\hspace{-0.7mm} =\hspace{-0.7mm} f^{k}$. The maximality
of $i, j$ together with Theorem \ref{Zieve-Muller} leads to
$k\hspace{-0.7mm} \leq\hspace{-0.7mm} \max{(8, 2\mathopen
+\mathclose 2\log_{2}{\deg{f}})}$.
\end{proof}

As a further corollary we have
\begin{corollary}\label{corollary of rigidity}
Let $f\hspace{-0.7mm} \in\hspace{-0.7mm} \mathrm{End}(\EE)\mathopen
\setminus\mathclose \mathrm{Aut}(\EE)$ that there exists no
$\iota\hspace{-0.7mm} \in\hspace{-0.7mm} \mathrm{Aut}{(\EE)}$ for
which $\iota\hspace{-0.7mm} \circ\hspace{-0.7mm} f\hspace{-0.7mm}
\circ\hspace{-0.7mm} \iota\!^{-\!1}\hspace{-0.7mm} =\hspace{-0.7mm}
z^{\deg{f}}$. Then there is a finite subset $\calS$ such that if two
finite Blaschke products $r$ and $s$ satisfy $r\hspace{-0.7mm}
\circ\hspace{-0.7mm} s\hspace{-0.7mm} =\hspace{-0.7mm} f^{d}$ then
the following assertions
\begin{enumerate}
\item[(i)] either there exists $h\hspace{-0.7mm} \in\hspace{-0.7mm}\mathrm{End}(\EE)$ for
which $r\hspace{-0.7mm} =\hspace{-0.7mm} f\hspace{-0.7mm}
\circ\hspace{-0.7mm} h$ or there exists $\iota\hspace{-0.7mm}
\in\hspace{-0.7mm} \mathrm{Aut}{(\EE)}$ for which $r\hspace{-0.7mm}
\circ\hspace{-0.7mm} \iota \in \calS$;

\item[(ii)] either there exists $h\hspace{-0.7mm} \in\hspace{-0.7mm}\mathrm{End}(\EE)$ for
which $s\hspace{-0.7mm} =\hspace{-0.7mm} h\hspace{-0.7mm}
\circ\hspace{-0.7mm} f$ or there exists $\iota\hspace{-0.7mm}
\in\hspace{-0.7mm} \mathrm{Aut}{(\EE)}$ for which
$\iota\hspace{-0.7mm} \circ\hspace{-0.7mm} s \in \calS$.
\end{enumerate}
are satisfied.
\end{corollary}
\begin{proof}
We only prove the first assertion as a similar argument applies to
the second one. If there exists no $h\hspace{-0.7mm}
\in\hspace{-0.7mm} \mathrm{End}(\EE)$ for which  $r\hspace{-0.7mm}
=\hspace{-0.7mm} f\hspace{-0.7mm} \circ\hspace{-0.7mm} h$, then
Corollary \ref{Zieve-Muller, corallary} implies that $r$ is a left
factor of $f^k$ for some $k\hspace{-0.7mm} \leq\hspace{-0.7mm}
\max{(8, 2\mathopen +\mathclose 2\log_{2}{\deg{f}})}$. Up to
associations there are only finitely many such factors.
\end{proof}

\section{Speciality of monoid factorizations}\label{sectionbilutichy}

If the fiber product $\PP^1\times_{f, g}\PP^1$ admits special
arithmetical or geometric properties for rational functions $f$ and
$g$, then $f$ and $g$ tend to have very special factorizations in
$(\mathrm{End}(\PP^1), \circ)$. We shall call this sort of facts the
speciality of monoid factorizations.  The goal of this section is to
obtain speciality of factorizations of $(\mathrm{End}(\EE), \circ)$,
under assumptions of finiteness of rational points. We begin with
recalling the complex analytic version of famous Bilu-Tichy
criterion\,(cf.\,\cite{BT00}).

\begin{theorem}[Bilu-Tichy]\label{bilutichy}
Let $f$ and $g$ be nonlinear polynomials that
$\CC\mathopen\times_{f, \, g}\mathclose\CC$ has a Siegel factor.
Then $f$ and $g$ admit the following factorizations
\begin{eqnarray*}
f\hspace{-0.7mm} =\hspace{-0.7mm} e\hspace{-0.7mm}
\circ\hspace{-0.7mm} f_1\hspace{-0.7mm} \circ\hspace{-0.7mm}
\varepsilon, \ \ \  g\hspace{-0.7mm} =\hspace{-0.7mm}
e\hspace{-0.7mm} \circ\hspace{-0.7mm} g_1\hspace{-0.7mm}
\circ\hspace{-0.7mm} \epsilon
\end{eqnarray*}
in $(\mathrm{End}(\CC), \circ)$ where $\{\varepsilon,
\epsilon\}\hspace{-0.7mm} \subset\hspace{-0.7mm} \mathrm{Aut}(\CC)$
and there exist $\{m,n\}\hspace{-0.7mm} \subset\hspace{-0.7mm} \NN$
together with $p \hspace{-0.7mm} \in\hspace{-0.7mm}
\CC[z]\mathopen\setminus\mathclose\{0\}$ such that $\{f_1, g_1\}$
falls into one of the following cases:
\begin{enumerate}
\item[(i)] $\{z^m, z^rp(z)^m\}$ with $r \ge 1$ and $(r, m) = 1;$

\item[(ii)] $\{z^2, (z^2 + 1) p(x)^2\};$

\item[(iii)] $\{T_m, T_n\}$ with $m \ge 3, n \ge 3$ and $(m, n) = 1;$

\item[(iv)] $\{T_m, -T_n\}$ with $m \ge 3, n \ge 3$ and $(m, n) > 1;$

\item[(v)] $\{(z^2\hspace{-0.5mm} -\hspace{-0.5mm} 1)^3,
3z^4\hspace{-0.5mm} -\hspace{-0.5mm} 4 z^3\}.$
\end{enumerate}
\end{theorem}

In this section we shall prove

\begin{theorem}\label{Faltings factor}
If the curve $\PP^1\mathopen \times_{f,\, g}\mathclose \PP^1$
defined by $\{f, g\} \hspace{-0.7mm} \subset\hspace{-0.7mm}
\mathrm{End}(\EE)$ has a Faltings factor then $f$ and $g$ admit the
following factorizations
\begin{eqnarray*}
f\hspace{-0.7mm} =\hspace{-0.7mm}e\hspace{-0.7mm}
\circ\hspace{-0.7mm} f_1\hspace{-0.7mm} \circ\hspace{-0.7mm}
\varepsilon, \ \ \ g\hspace{-0.7mm} =\hspace{-0.7mm}
e\hspace{-0.7mm} \circ\hspace{-0.7mm} g_1\hspace{-0.7mm}
\circ\hspace{-0.7mm} \epsilon
\end{eqnarray*}
in $(\mathrm{End}(\EE), \circ)$ where $\{\varepsilon, \epsilon\}
\hspace{-0.7mm} \subset\hspace{-0.7mm}\mathrm{Aut}(\EE)$ and there
exist positive integers $m, n$ and $p\hspace{-0.7mm}
\in\hspace{-0.7mm}\mathrm{End}(\EE)\hspace{-0.7mm}
\cup\hspace{-0.7mm}\{1\}$ such that $\{f_1, g_1\}$ falls into one of
the following cases:
\begin{enumerate}
\item[(i)] $\{z^m, z^rp(z)^m\}$ with $r \ge 1$ and $(r, m) = 1;$
\item[(ii)] $\{z^2, z (z\hspace{-0.5mm} -\hspace{-0.5mm}
a)/(1\hspace{-0.5mm} -\hspace{-0.5mm} \overline{a}z)p(z)^2\}$ with
$a\hspace{-0.7mm}  \in\hspace{-0.7mm}
\EE\mathopen\setminus\mathclose\{0\};$
\item[(iii)] $\{\calT_{m,\, nt}, \calT_{n,\, mt}\}$ with
$t\hspace{-0.5mm}
>\hspace{-0.5mm} 0,$ $m \ge 3, n \ge 3$ and $(m, n)\hspace{-0.5mm}
=\hspace{-0.5mm} 1;$
\item[(iv)] $\{\calT_{m, \, nt}, -\hspace{-0.5mm}\calT_{n,\, mt}\}$
with $t\hspace{-0.5mm} >\hspace{-0.5mm} 0, $ $m \ge 3, n \ge 3$ and
$(m, n)\hspace{-0.5mm}
>\hspace{-0.5mm} 1;$
\item[(v)] $\{((z^2\hspace{-0.5mm} -\hspace{-0.5mm}
a^2)/(1\hspace{-0.5mm} -\hspace{-0.5mm} \overline{a}^2z^2))^3, z^3
(z\hspace{-0.5mm} -\hspace{-0.5mm} b)/(1\hspace{-0.5mm}
-\hspace{-0.5mm} \overline{b}z)\}$ where $a, b$ are points in $\EE$
and $a, b,  \ol{a\bPh}, \overline{b}$ satisfy an algebraic relation.
\end{enumerate}
\end{theorem}
\begin{proof}
Follow the notation used before Lemma \ref{Siegel-Faltings} and
write $\overline{f}\hspace{-0.7mm} :=\hspace{-0.7mm}(j_1,\, i)_{*}f,
\, \overline{g}\hspace{-0.7mm} :=\hspace{-0.7mm}(j_2,\, i)_{*}g$. By
definition we have $ f\hspace{-0.7mm} =\hspace{-0.5mm} i\!^{-\!1}
\hspace{-0.7mm}\circ\hspace{-0.5mm} \overline{f}\hspace{-0.7mm}
\circ\hspace{-0.5mm} j_1,\, g\hspace{-0.7mm} =\hspace{-0.5mm}
i\!^{-\!1}\hspace{-0.7mm} \circ\hspace{-0.5mm}
\overline{g}\hspace{-0.7mm} \circ\hspace{-0.5mm} j_2 $, which also
means that $j\hspace{-0.1mm}^{-\!1}_1$ is a $\overline{f}$-lifting
of $i\!^{-\!1}$. If $\PP^1\mathopen \times_{f,\, g}\mathclose \PP^1$
has a Faltings factor then by Lemma \ref{Siegel-Faltings} the curve
$\CC\mathopen \times_{\overline{f}, \, \overline{g\bPh}}\,\mathclose
\CC$ has a Siegel factor. By Bilu-Tichy's Criterion there exist
$\{{\displaystyle \overline{\varepsilon}, \overline{\epsilon}}\}
\hspace{-0.7mm} \subset\hspace{-0.7mm} \mathrm{Aut}(\CC)$ such that
$\overline{f}, \, \overline{g\bPh}$ admit one of the following
factorizations in $(\mathrm{End}(\CC),\circ)$:

{\it (i)}  $\overline{f} \hspace{-0.7mm} =\hspace{-0.7mm}
\overline{e\bPh}\hspace{-0.7mm} \circ z^m\hspace{-0.7mm}
\circ\hspace{-0.7mm} \overline{\varepsilon\bPh}, \ \ \
\overline{g\bPh}\hspace{-0.7mm} =\hspace{-0.7mm}
\overline{e\bPh}\hspace{-0.7mm} \circ\hspace{-0.7mm}
z^r\overline{p\bPh}(z)^m\hspace{-0.7mm} \circ\hspace{-0.7mm}
\overline{\epsilon\bPh}$.

Let $i_1$ be a $\overline{e}$-lifting of $i\!^{-\!1}$ and $i_2$ a
$z^m$-lifting of $i_1$. By Proposition \ref{deformation of one
factorization, converse} and an induction argument,
$j\hspace{-0.1mm}^{-\!1}_1$ is a $\overline{\varepsilon}$-lifting of
$i_2$, and then $ f\mathopen =\mathclose e\mathopen \circ\mathclose
f_1\mathopen \circ\mathclose \varepsilon $ is a relation in
$(\mathrm{End}(\EE),\circ)$ where $e, f_1$ and $\varepsilon$ are
obtained by the following commutative diagram.
\begin{displaymath}
\xymatrix{
\CC \ar@/^1pc/[rrr]^{\overline{f}} \ar[d]^{j\hspace{-0.1mm}^{-\!1}_1}  \ar[r]_{\overline{\varepsilon}}& \CC \ar[d]^{i_2} \ar[r]_{z^m} &  \CC \ar[d]^{i_1} \ar[r]_{\overline{e}} &  \CC \ar[d]^{i\!^{-\!1}}\\
\EE \ar@/_1pc/[rrr]_{f}   \ar[r]^{\varepsilon}    & \EE \ar[r]^{f_1}
& \EE  \ar[r]^{e} & \EE}
\end{displaymath}
Similarly if $i_2'$ is a $z^r\overline{p}(z)^m$-lifting of $i_1$
then $ g \mathopen =\mathclose e\mathopen \circ\mathclose
g_1\mathopen \circ\mathclose \epsilon $ is also a relation in
$(\mathrm{End}(\EE),\circ)$ according to the following commutative
diagram.
\begin{displaymath}
\xymatrix{
\CC \ar@/^1pc/[rrr]^{\overline{g}} \ar[d]^{j\hspace{-0.1mm}^{-\!1}_2}  \ar[r]_{\overline{\epsilon}}& \CC \ar[d]^{i_2'} \ar[r]_{z^r\overline{p}(z)^m} &  \CC \ar[d]^{i_1} \ar[r]_{\overline{e}} &  \CC \ar[d]^{i\!^{-\!1}}\\
\EE \ar@/_1pc/[rrr]_{g}   \ar[r]^{\epsilon}    & \EE \ar[r]^{g_1} &
\EE  \ar[r]^{e} & \EE}
\end{displaymath}
Write $\mathfrak{p}\hspace{-0.6mm} =\hspace{-0.6mm} i_1(0),
\mathfrak{r}\hspace{-0.6mm} =\hspace{-0.6mm} i_2(0)$ and
$\mathfrak{q}\hspace{-0.6mm} =\hspace{-0.6mm} i_2'(0)$. The map
$f_1$ is totally ramified over $\mathfrak{p}$ with $\mathfrak{r}$
above, and $ (g_1)_{\mathfrak{p}}\hspace{-0.6mm} \equiv
\hspace{-0.6mm} r (\mathfrak{q}) \pmod{m} $. Choosing suitable
$\iota_i$ in $\mathrm{Aut}(\EE)$ and substituting \bgroup
\arraycolsep=1pt
\begin{eqnarray}\label{replacement}
\nonumber e &\mapsto& e\hspace{-0.7mm} \circ\hspace{-0.7mm}
\iota\hspace{-0.1mm}^{-\!1}_1\hspace{-0.7mm},\\
\nonumber  \varepsilon &\mapsto& \iota_2\hspace{-0.7mm}
\circ\hspace{-0.7mm}
\varepsilon,\\
\epsilon \, &\mapsto& \iota_3\hspace{-0.7mm} \circ\hspace{-0.7mm}
\epsilon,\\
\nonumber  f\hspace{-0.6mm}_1 &\mapsto&
\iota\hspace{-0.15mm}_1\hspace{-0.7mm} \circ\hspace{-0.7mm}
f_1\hspace{-0.7mm} \circ\hspace{-0.7mm}
\iota\hspace{-0.1mm}^{-\!1}_2\hspace{-0.7mm},\\
\nonumber  g\hspace{-0.2mm}_1 &\mapsto&
\iota\hspace{-0.15mm}_1\hspace{-0.7mm} \circ\hspace{-0.7mm}
g_1\hspace{-0.7mm} \circ\hspace{-0.7mm}
\iota\hspace{-0.1mm}^{-\!1}_3\hspace{-0.7mm}
\end{eqnarray}
\egroup \noindent we may assume that $\mathfrak{p}\hspace{-0.7mm}
=\hspace{-0.7mm} \mathfrak{r} \hspace{-0.7mm} =\hspace{-0.7mm}
\mathfrak{q}\hspace{-0.7mm} =\hspace{-0.7mm} 0$, and this leads to
the desired assertion.

{\it (ii)} $\overline{f}\hspace{-0.7mm} =\hspace{-0.7mm}
\overline{e}\hspace{-0.7mm} \circ\hspace{-0.6mm} z^2\hspace{-0.7mm}
\circ\hspace{-0.7mm} \overline{\varepsilon}, \ \ \
\overline{g}\hspace{-0.7mm} =\hspace{-0.6mm}
\overline{e}\hspace{-0.7mm} \circ\hspace{-0.7mm}
(z^2\hspace{-0.7mm}+\hspace{-0.7mm}1)p(z)^2\hspace{-0.7mm}
\circ\hspace{-0.7mm} \overline{\epsilon}.$

By arguments similar to that in the proof of previous case we obtain
the following relations $f\hspace{-0.7mm} =\hspace{-0.7mm}
e\hspace{-0.7mm} \circ\hspace{-0.7mm} f_1\hspace{-0.7mm}
\circ\hspace{-0.7mm} \varepsilon$, $g\hspace{-0.7mm}
=\hspace{-0.7mm} e\hspace{-0.7mm} \circ\hspace{-0.7mm}
g_1\hspace{-0.7mm} \circ\hspace{-0.7mm} \epsilon$ in
$(\mathrm{End}(\EE),\circ)$ in which $f_1$ is totally ramified over
some $\mathfrak{p}$ and $ (g_1)_{\mathfrak{p}}\hspace{-0.7mm}
\equiv\hspace{-0.7mm} (\mathfrak{q})\hspace{-0.7mm} +\hspace{-0.7mm}
(\mathfrak{r}) \pmod{2} $ for some distinct points $\mathfrak{q},
\mathfrak{r}$ in $\EE.$ Choosing suitable $\iota_i$ in
$\mathrm{Aut}(\EE)$ and substituting as in (\ref{replacement}) we
may assume that $\mathfrak{p} \hspace{-0.7mm} =\hspace{-0.7mm}
\mathfrak{q}\hspace{-0.7mm} =\hspace{-0.7mm} 0,
\mathfrak{r}\hspace{-0.7mm} =\hspace{-0.7mm} a$, and this implies
our desired assertion.

{\it (iii)} $\overline{f}\hspace{-0.7mm} =\hspace{-0.7mm}
\overline{e} \hspace{-0.7mm}\circ\hspace{-0.7mm} T_m\hspace{-0.7mm}
\circ\hspace{-0.7mm} \overline{\varepsilon}, \ \ \
\overline{g}\hspace{-0.7mm} =\hspace{-0.7mm}
\overline{e}\hspace{-0.7mm} \circ\hspace{-0.7mm} T_n\hspace{-0.7mm}
\circ\hspace{-0.7mm} \overline{\epsilon}$ with $(m,
n)\hspace{-0.7mm} =\hspace{-0.7mm} 1.$

By arguments similar to that in the proof of case {\it (i)} we may
obtain the following relations $f\hspace{-0.7mm} =\hspace{-0.7mm}
e\hspace{-0.7mm} \circ\hspace{-0.7mm} f_1\hspace{-0.7mm}
\circ\hspace{-0.7mm} \varepsilon$ and $g\hspace{-0.7mm}
=\hspace{-0.7mm} e\hspace{-0.7mm} \circ\hspace{-0.7mm}
g_1\hspace{-0.7mm} \circ\hspace{-0.7mm} \epsilon$ in
$(\mathrm{End}(\EE),\circ)$ where $f_1, g_1$ are both unramified
outside $\{\mathfrak{p}, \mathfrak{q}\}$ for some distinct points
$\mathfrak{p}, \mathfrak{q}$ in $\EE$ and their monodromy are
Chebyshev representation. By Proposition \ref{nomalized
Chebyshev-Blaschke products,uniqueness}, after substituting as in
(\ref{replacement}) for suitable $\iota_i$ chosen from
$\mathrm{Aut}(\EE)$ we will have $f_1\hspace{-0.7mm}
=\hspace{-0.7mm} \mathcal{T}_{m, \, nt}$ and $g_1\hspace{-0.7mm}
=\hspace{-0.7mm} \mathcal{T}_{n, \, mt}$ as desired.

{\it (iv)} $\overline{f}\hspace{-0.7mm} =\hspace{-0.7mm}
\overline{e}\hspace{-0.7mm} \circ\hspace{-0.7mm} T_m\hspace{-0.7mm}
\circ\hspace{-0.7mm} \overline{\varepsilon}, \ \ \
\overline{g}\hspace{-0.7mm} =\hspace{-0.7mm}
\overline{e}\hspace{-0.7mm} \circ\hspace{-0.7mm}
-\!T_n\hspace{-0.6mm} \circ\hspace{-0.7mm} \overline{\epsilon}$ with
$(m, n)\hspace{-0.6mm}
>\hspace{-0.6mm} 1.$

We may apply arguments similar to that in the proof of Case {\it
(iii)}.

{\it (v)} $\overline{f}\hspace{-0.7mm} =\hspace{-0.7mm}
\overline{e}\hspace{-0.7mm} \circ\hspace{-0.7mm} (z^2\hspace{-0.5mm}
-\hspace{-0.7mm} 1)^3\hspace{-0.7mm} \circ\hspace{-0.7mm}
\overline{\varepsilon}, \ \ \ \overline{g}\hspace{-0.7mm}
=\hspace{-0.7mm} \overline{e}\hspace{-0.7mm} \circ\hspace{-0.7mm}
(3z^4\hspace{-0.7mm} -\hspace{-0.7mm} 4z^3)\hspace{-0.7mm}
\circ\hspace{-0.7mm} \overline{\epsilon}$.

We first notice that $(z^2\hspace{-0.7mm} -\hspace{-0.7mm} 1)^3$
takes $-1$ and $0$ as critical values, $\pm 1$ over 0 and $0$ over
$-\!1$ with ramification index $e_{\pm1}\hspace{-0.7mm}
=\hspace{-0.7mm} 3$ and $e_0\hspace{-0.7mm} =\hspace{-0.7mm} 2$.
Moreover $3z^4\hspace{-0.7mm} -\hspace{-0.7mm} 4z^3$ takes also $-1$
and $0$ as critical values, 0 over 0 and 1 over $-\!1$ with
$e_0\hspace{-0.5mm} =\hspace{-0.5mm} 3$ and $e_1\hspace{-0.5mm}
=\hspace{-0.5mm} 2$. By arguments similar to that in the proof of
case {\it (i)} we obtain the following relations $f\hspace{-0.7mm}
=\hspace{-0.7mm} e\hspace{-0.7mm} \circ\hspace{-0.7mm}
f_1\hspace{-0.7mm} \circ\hspace{-0.7mm} \varepsilon$,
$g\hspace{-0.7mm} =\hspace{-0.7mm} e\hspace{-0.7mm}
\circ\hspace{-0.7mm} g_1\hspace{-0.7mm} \circ\hspace{-0.7mm}
\epsilon$ in $(\mathrm{End}(\EE),\circ)$, where $f_1$ admits two
points $\mathfrak{q}, \mathfrak{r}$ ramified over some point
$\mathfrak{p}$ with $e_{\mathfrak{q}}\hspace{-0.7mm}
=\hspace{-0.7mm} e_{\mathfrak{r}}\hspace{-0.7mm} =\hspace{-0.7mm} 3$
and $g_1$ admits a point $\mathfrak{s}$ ramified over $\mathfrak{p}$
with $e_{\mathfrak{s}}\hspace{-0.7mm} =\hspace{-0.7mm}3$. Making a
replacement as in (\ref{replacement}) for well-chosen $\iota_i$ in
$\mathrm{Aut}(\EE)$ we may assume that $\mathfrak{p}\hspace{-0.7mm}
=\hspace{-0.7mm} \mathfrak{s}\hspace{-0.7mm} =\hspace{-0.7mm} 0,
\mathfrak{q}\hspace{-0.7mm} =\hspace{-0.7mm} -\! \mathfrak{r}$ which
gives the desired $f_1$ and $g_1$. The algebraic relation is given
by the coincidence of another critical value of $f_1$ and $g_1$ .
\end{proof}

In case{\it (i)} if $g_1$ is totally ramified with $m
\hspace{-0.7mm} \ge \hspace{-0.7mm} 2$ then one checks readily that
$g_1$ is ramified over $0$. After modifying $\epsilon$ we can assume
$\{f_1, g_1\} \hspace{-0.7mm} = \hspace{-0.7mm} \{z^m, z^t\}$.

\section{A result on heights}\label{section heights}
In this section we shall prove Theorem \ref{equal degree, general
case} by comparing the logarithmic naive height and Call-Siverman's
canonical height. The key ingredient of the proof is a recent
theorem of M. Baker \cite{B09}.

Given a global field $E$ we write $M_{E}$ for the set of normalized
absolute values. Because the Picard group of $\PP^1$ is $\ZZ$, it is
clear that for any $\iota \in \mathrm{Aut}_{\overline{E}}(\PP^1)$
there exists a positive constant $c$ such that for all $x$ in
$\PP^1(\overline{E})$ we have $| h(\iota(x)) - h(x) | \leq c$.

Given $f\in \mathrm{End}(\PP^1)$ that is defined over $E$ the
canonical height $\hat{h}_f(z)$ satisfies $\hat{h}_{f}(f^k(z)) =
(\deg{f})^k\, \hat{h}_f(z)$, $|h(z) - \hat{h}_f(z)|$ is uniformly
bounded and in case $E$ is a number field then $z$ is preperiodic if
and only if $\hat{h}_f(z) = 0.$

Let $E$ be a function field. We call $g\hspace{-0.7mm}
\in\hspace{-0.7mm} E(x)$ {\it isotrivial} if there is a finite
extension $E'$ of $E$ and $\iota\hspace{-0.7mm} \in\hspace{-0.7mm}
\mathrm{Aut}_{E'}(\PP)$ such that $\iota\hspace{-0.7mm}
\circ\hspace{-0.7mm} g\hspace{-0.7mm} \circ\hspace{-0.7mm}
\iota\!^{-\!1}$ is defined over the field of constants.

\begin{theorem}[M. Baker]
If $E$ is a function field and if $f\hspace{-0.7mm}
\in\hspace{-0.7mm}E(\PP^1)\mathopen\setminus\mathclose
\mathrm{Aut}_{E}(\PP^1)$ is non-isotrivial then a point
$z\hspace{-0.7mm} \in\hspace{-0.7mm}\PP^1(\overline{E})$ is
preperiodic if and only if $\hat{h}_f(z) \hspace{-0.7mm}
=\hspace{-0.7mm} 0$.
\end{theorem}

This theorem is crucial for the proof of the following

\begin{theorem}\label{equal degree, general case}
Let $\{f, g\} \hspace{-0.7mm} \subset\hspace{-0.7mm}\CC(z)$  and let
$\{x_0, y_0\}\hspace{-0.7mm} \subset\hspace{-0.7mm}\PP^1$. If
$\calO_{f\mathopen \times\mathclose g}(x_0, y_0)$ has infinitely
many points on the diagonal of $\PP^1\mathopen\times\mathclose\PP^1$
then $\deg{f}\hspace{-0.7mm} =\hspace{-0.7mm} \deg{g}$.
\end{theorem}



In \cite[p.478]{GTZ08} the authors announced that they can prove
Theorem \ref{equal degree, general case} for polynomials via
Benedetto's theorem together with many other results from polynomial
dynamics. Based on some idea of Ghioca-Tucker-Zieve we invoke only
M. Baker's theorem to prove the above theorem by induction on the
transcendental degree of a field of definition of $f, g, x_0, y_0$
over $\QQ$. We start with

\begin{lemma}\label{equal degree, number field case}
Let $k$ be a number field, $\{f, g\}\hspace{-0.7mm}
\subset\hspace{-0.7mm}k(\PP^1)$ and $\{x_0, y_0\}\hspace{-0.7mm}
\subset\hspace{-0.7mm}\PP^1(k)$. If the orbit $\calO_{f\mathopen
\times\mathclose g}(x_0,\, y_0)$ has infinitely many points on the
diagonal then $\deg{f}\hspace{-0.7mm} =\hspace{-0.7mm} \deg{g}$.
\end{lemma}
\begin{proof}
Otherwise we assume that $\deg{f}\hspace{-0.7mm} <\hspace{-0.7mm}
\deg{g}$ and that there exists $x$ in $k$ for which
$\calO_{f\mathopen \times\mathclose g}(x,\, x)$ has infinitely many
points on the diagonal. Note that $x$ is not a preperiodic point of
$g$, as otherwise $\calO_{f\mathopen \times\mathclose g}(x,\, x)$
has at most finitely many points on the diagonal. This leads to
$\hat{h}_g(x)\hspace{-0.7mm} >\hspace{-0.7mm} 0$. By properties of
heights we have
\begin{eqnarray*}
h(g^m(x)) \gg_m  \deg^m{g}.
\end{eqnarray*}
If $f\hspace{-0.7mm} \not \in\hspace{-0.7mm}
\mathrm{Aut}_{k}(\PP^1)$ then
\begin{eqnarray*}
h(f^m(x)) \ll_m \deg^m{f},
\end{eqnarray*}
and if $f$ is in $\mathrm{Aut}_{k}(\PP^1)$ then
\begin{eqnarray*}
h(f^m(x))  \ll_m \, m. \ \ \  \ \ \
\end{eqnarray*}
 By comparing the heights of $f^m(x)$ and
$g^m(x)$ we conclude that there are only finitely many $m$ for which
$f^m(x)\hspace{-0.7mm} =\hspace{-0.7mm} g^m(x)$. This contradicts
our assumption.
\end{proof}

To prove Theorem \ref{equal degree, general case} we use Lemma
\ref{equal degree, number field case} and the technique of
specialization.

\begin{proofof}{Proof of Theorem \ref{equal degree, general
case}} For the same reason as in the proof of Lemma \ref{equal
degree, number field case} we may assume $x_0\hspace{-0.7mm}
=\hspace{-0.7mm} y_0\hspace{-0.7mm} =\hspace{-0.7mm} x$,
$\deg{f}\hspace{-0.7mm} <\hspace{-0.7mm} \deg{g}$ and $x$ is
preperiodic for neither  $f$ nor $g$. Objets $f, g$ and $x$ are all
defined over a field $k$ of finite type over $\QQ$ and we continue
with the proof by induction on $tr. deg(k/\QQ)$.  If $tr.
deg(k/\QQ)\hspace{-0.7mm} =\hspace{-0.7mm} 0$ then it reduces to
Lemma \ref{equal degree, number field case}. Let $s$ be a positive
integer greater that the claim holds as long as $tr.
deg(K/\QQ)\hspace{-0.7mm} \leq\hspace{-0.7mm} s\mathopen -\mathclose
1$, then we will prove it for $tr. deg(K/\QQ)\hspace{-0.7mm}
=\hspace{-0.7mm} s$. Choose a subfield $k'$ of $k$ such that $tr.
deg(k/k')\hspace{-0.7mm} =\hspace{-0.7mm} 1$ and then $k$ is the
function field of a curve $X$ defined over $k'$. Now we restrict our
attention to $k\mathopen \times_{k'}\mathclose {\overline{k'} /
\overline{k'}}$ instead of $k / k'$. If $g$ is not isotrivial then
we also have $\hat{h}_g(x)\hspace{-0.7mm}
>\hspace{-0.7mm} 0$ by M. Baker's theorem and the argument in the proof of Lemma \ref{equal degree,
number field case} still works. Now we assume $g$ is isotrivial.
After a conjugation by a linear fractional transformation we may
assume $g$ is defined over $\overline{k'}$. Now we fall into one of
the following two cases:

\noindent {\it Case {\it (i)}}, $x\hspace{-0.7mm} \in\hspace{-0.7mm}
\overline{k'}$.

We choose $\alpha$ in $X(\overline{k'})$ at which $f$ has good
reduction and consider the reduction triple $f_{\alpha},\,
g_{\alpha}\hspace{-0.7mm} =\hspace{-0.7mm} g,
x_{\alpha}\hspace{-0.7mm} =\hspace{-0.7mm} x$. By assumption $x$ is
not preperiodic for $g$ and therefore $x_{\alpha}$ is not
preperiodic for $g_{\alpha}$. This means that
$\calO_{f_{\alpha}\mathopen \times\mathclose g_{\alpha}}(x_{\alpha},
x_{\alpha})$ has infinitely many points on the diagonal and we are
done by the induction assumption.

\noindent {\it Case {\it (ii)}}, $x\hspace{-0.7mm}
\not\in\hspace{-0.7mm} \overline{k'}$.

We will give two alternative arguments. For the first proof we
notice that $x$ is a function of positive degree $d$ on
$X(\overline{k'})$ and therefore $g^m(x)$ is a function of degree $d
\deg^m{g}$ on $X$. Moreover by induction it follows easily that
there exists a natural number $e$ such that for all positive integer
$m$ the function $f^{m}(x)$ is of degree at most $d
\deg^m{f}\hspace{-0.7mm} +\hspace{-0.7mm} em \deg^m{f}$. We obtain a
contradiction by comparing the degrees of $f^{m}(x)$ and of
$g^{m}(x)$. For the second proof we notice that $g$ is a function in
$\overline{k'}(z)$ and there exists $q$ in $\overline{k'}$ such that
$q$ is not preperiodic for $g$. Let $\alpha$ be a point in
$X(\overline{k'})$ for which $x_{\alpha}$ equals $q$. We do the
reduction at $\alpha$ and then we complete the proof by the
induction assumption.
\end{proofof}

\section{Proof of Theorem \ref{disk}}\label{section proof}
The comparison of rigidity with speciality of monoid factorizations
is implicitly one major originality of \cite{GTZ08} and
\cite{GTZ10}, where the authors used the rigidity\,(cf.\,Ritt
\cite{R22}) and the speciality\,(cf.\,Bilu-Tichy \cite{BT00}) of
factorizations of $(\mathrm{End}(\CC), \circ)$ to study the dynamics
of polynomials. We have obtained the rigidity of factorizations of
$(\mathrm{End}(\EE), \circ)$ in Theorem \ref{NgWangmonoid} and
Proposition \ref{rigidity of Blaschke}\,(\,based on a joint work
with Ng), as well as the corresponding speciality result in Theorem
\ref{Faltings factor}. The former relies on action of fundamental
groups, while the latter is governed by finiteness of rational
points. In this section we shall adopt the strategy of
Ghioca-Tucker-Zieve and work on $(\mathrm{End}(\EE), \circ)$. One
more difficulty in carrying their method in our context is the
management of elliptic rational functions.

\begin{proposition}\label{finiteness}
If at least one of $\{f, g\}\hspace{-0.7mm} \subset\hspace{-0.7mm}
\mathrm{End}(\EE)$ is not totally ramified then the equation
$\epsilon\hspace{-0.7mm} \circ\hspace{-0.7mm} f\hspace{-0.7mm}  =
\hspace{-0.7mm} g\hspace{-0.7mm} \circ\hspace{-0.7mm}  \varepsilon$
has only finitely many solutions $\{\epsilon,
\varepsilon\}\hspace{-0.7mm}
\subset\hspace{-0.7mm}\mathrm{Aut}(\EE)$.
\end{proposition}
\begin{proof}
We assume that $f$ is not totally ramified and then the degree of
$|\mathfrak{D}_f|$ and of $\mathfrak{d}_f$ are at least two. If
$\{\epsilon, \varepsilon\}\hspace{-0.7mm}
\subset\hspace{-0.7mm}\mathrm{Aut}(\EE)$ gives a solution to
$\epsilon\hspace{-0.7mm} \circ\hspace{-0.7mm} f\hspace{-0.7mm}
=\hspace{-0.7mm} g\hspace{-0.7mm} \circ\hspace{-0.7mm} \varepsilon$
then $\epsilon$\,(resp.\,$\varepsilon$) induces a bijection from
$\mathfrak{d}_f$ to $\mathfrak{d}_g$\,(resp.\,from
$|\mathfrak{D}_f|$ to $|\mathfrak{D}_g|$). The natural map
\begin{eqnarray*}
\{(\epsilon, \varepsilon)\,|\, \epsilon\hspace{-0.7mm}
\circ\hspace{-0.7mm}  f\hspace{-0.7mm}  =\hspace{-0.7mm}
g\hspace{-0.7mm}  \circ\hspace{-0.7mm}  \varepsilon\} \rightarrow
\mathrm{Hom}(\mathfrak{d}_f, \mathfrak{d}_g)\mathopen
\times\mathclose \mathrm{Hom}(|\mathfrak{D}_f|,  |\mathfrak{D}_g|)
\end{eqnarray*}
is injective since the only element in $\mathrm{Aut}(\EE)$ that
fixes at least two points must be the identity map. This leads to
the desired assertion.
\end{proof}

Finite Blaschke products $f, g$ are called {\it commensurable} in
$(\mathrm{End}(\EE),\circ)$ if for any positive integer $m$ there
exist $\{h_1, h_2\}
\hspace{-0.7mm}\subset\hspace{-0.7mm}\mathrm{End}(\EE)$ and positive
integer $n$ such that
\begin{eqnarray*} f^{n}\hspace{-0.7mm} =\hspace{-0.7mm} g^{m}\hspace{-0.7mm} \circ\hspace{-0.7mm} h_1, \ \ \
g^{n}\hspace{-0.7mm} =\hspace{-0.7mm} f^{m}\hspace{-0.7mm}
\circ\hspace{-0.7mm} h_2.
\end{eqnarray*}
\begin{lemma}\label{degf=degg}
Let $f\hspace{-0.7mm} \in\hspace{-0.7mm} \mathrm{End}(\EE)\mathopen
\setminus\mathclose \mathrm{Aut}(\EE)$ and $\iota\hspace{-0.7mm}
\in\hspace{-0.7mm} \mathrm{Aut}(\EE)$. If there are infinitely many
$n$ such that $f^{n}\hspace{-0.7mm} =\hspace{-0.7mm}
(f\hspace{-0.7mm}\circ\hspace{-0.7mm}\iota)^{n}
\hspace{-0.7mm}\circ\hspace{-0.7mm} \iota_n$ for some
$\iota_n\hspace{-0.7mm}\in\hspace{-0.7mm} \mathrm{Aut}(\EE)$ then
one of the following assertions
\begin{enumerate}
\item[(i)] there exists $k\hspace{-0.7mm}
\in\hspace{-0.7mm}\NN$ for which $ f^{k}\hspace{-0.7mm}
=\hspace{-0.7mm} (f\hspace{-0.7mm}\circ\hspace{-0.7mm}\iota)^{k} $.
\item[(ii)]
there exist $\{\mu, \rho\}\hspace{-0.7mm} \subset\hspace{-0.7mm}
\TT$ and $\epsilon\hspace{-0.7mm}
\in\hspace{-0.7mm}\mathrm{Aut}(\EE)$ such that $ f\hspace{-0.7mm}
=\hspace{-0.7mm} \epsilon\hspace{-0.7mm}\circ\hspace{-0.7mm}\mu\,
z^d\hspace{-0.7mm}\circ\hspace{-0.7mm}\epsilon^{-1}$ and
$\iota\hspace{-0.7mm} =\hspace{-0.7mm}
\epsilon\hspace{-0.7mm}\circ\hspace{-0.7mm}\rho\, z
\circ\hspace{-0.7mm}\epsilon^{-1}. $
\end{enumerate}
is satisfied.
\end{lemma}
\begin{proof}
By Corollary \ref{simple rigidity} applied to $f^{n}\hspace{-0.7mm}
=\hspace{-0.7mm} (f\hspace{-0.7mm} \circ\hspace{-0.7mm}
\iota)^{n}\hspace{-0.7mm} \circ\hspace{-0.7mm} \iota_n$, there exist
$\{\epsilon_n, \varepsilon_n\}\hspace{-0.7mm}
\subset\hspace{-0.7mm}\mathrm{Aut}(\EE)$ such that $
f\hspace{-0.7mm} \circ\hspace{-0.7mm} \iota\hspace{-0.7mm}
\circ\hspace{-0.7mm} \iota_n\hspace{-0.7mm} =\hspace{-0.7mm}
\epsilon_n\hspace{-0.7mm} \circ\hspace{-0.7mm} f$ and
$f\hspace{-0.7mm} \circ\hspace{-0.7mm} \iota\hspace{-0.7mm}
\circ\hspace{-0.7mm} f\hspace{-0.7mm} \circ\hspace{-0.7mm}
\iota\hspace{-0.7mm} \circ\hspace{-0.7mm} \iota_n\hspace{-0.7mm}
=\hspace{-0.7mm} \varepsilon_n\hspace{-0.7mm} \circ\hspace{-0.7mm}
f^{2}. $

\noindent {\it Case {\it (i)}}, $f$ is not totally ramified. By
Proposition \ref{finiteness} there exists $n\hspace{-0.7mm}
<\hspace{-0.7mm} m$ such that $\iota_n\hspace{-0.7mm}
=\hspace{-0.7mm} \iota_m$. This gives $ f^{m}\hspace{-0.7mm}
=\hspace{-0.7mm} (f\hspace{-0.7mm} \circ\hspace{-0.7mm}
\iota)^{m}\hspace{-0.7mm} \circ\hspace{-0.7mm}
\iota_n\hspace{-0.7mm} =\hspace{-0.7mm} (f\hspace{-0.7mm}
\circ\hspace{-0.7mm} \iota)^{m\mathopen -\mathclose
n}\hspace{-0.7mm} \circ\hspace{-0.7mm} (f\hspace{-0.7mm}
\circ\hspace{-0.7mm} \iota)^{n}\hspace{-0.7mm} \circ\hspace{-0.7mm}
\iota_n\hspace{-0.7mm} =\hspace{-0.7mm} (f\hspace{-0.7mm}
\circ\hspace{-0.7mm} \iota)^{m\mathopen -\mathclose
n}\hspace{-0.7mm} \circ\hspace{-0.7mm} f^{n} $ and therefore
$f^{m\mathopen-\mathclose n}\hspace{-0.7mm} =\hspace{-0.7mm}
(f\hspace{-0.7mm} \circ\hspace{-0.7mm} \iota)^{m\mathopen
-\mathclose n}$.

\noindent {\it Case {\it (ii)}}, $f^{2}$\,(equivalently
$f\hspace{-0.7mm} \circ\hspace{-0.7mm} \iota\hspace{-0.7mm}
\circ\hspace{-0.7mm} f$) is not totally ramified, a similar argument
works.

\noindent {\it Case {\it (iii)}}, $f, f^{2}$ and $f\hspace{-0.7mm}
\circ\hspace{-0.7mm} \iota\hspace{-0.7mm} \circ\hspace{-0.7mm} f$
are all totally ramified. Write  $\mathfrak{q}\hspace{-0.7mm}
=\hspace{-0.7mm} |\mathfrak{D}_f|$ and $\mathfrak{p}\hspace{-0.7mm}
=\hspace{-0.7mm} \mathfrak{d}_{f}$. Because
$f^2$\,(resp.\,$f\hspace{-0.7mm} \circ\hspace{-0.7mm}
\iota\hspace{-0.7mm} \circ\hspace{-0.7mm} f$) is totally ramified we
have $ \mathfrak{q}\hspace{-0.7mm} =\hspace{-0.7mm}
\mathfrak{p}$\,(resp.\,$\iota(\mathfrak{p})\hspace{-0.7mm}
=\hspace{-0.7mm} \mathfrak{p}$). Consequently there exist $\{\mu,
\rho\}\hspace{-0.7mm} \subset\hspace{-0.7mm}\TT$ such that
$f\hspace{-0.7mm} =\hspace{-0.7mm}
\iota_{\mathfrak{p}}\hspace{-0.7mm} \circ\hspace{-0.7mm} \mu\, z^d
\hspace{-0.7mm}\circ\hspace{-0.7mm} \iota_{-\mathfrak{p}}$ and
$\iota\hspace{-0.7mm} =\hspace{-0.7mm}
\iota_{\mathfrak{p}}\hspace{-0.7mm} \circ\hspace{-0.7mm} \rho\, z
\hspace{-0.7mm} \circ\hspace{-0.7mm} \iota_{-\mathfrak{p}}$.
\end{proof}

We first prove

\begin{proposition}\label{commensurable}
If $\{f, g\} \hspace{-0.7mm} \subset\hspace{-0.7mm}
\mathrm{End}(\EE)\mathopen \setminus\mathclose \mathrm{Aut}(\EE)$
are commensurable then either $f$ and $g$ have common iterations or
there exist $\iota\hspace{-0.7mm} \in\hspace{-0.7mm}
\mathrm{Aut}(\EE)$ and $\mu\hspace{-0.7mm} \in\hspace{-0.7mm}\TT$
such that
\begin{eqnarray*}
\iota\hspace{-0.7mm} \circ\hspace{-0.7mm} f\hspace{-0.7mm}
\circ\hspace{-0.7mm} \iota\!^{-\!1}\hspace{-0.7mm} =\hspace{-0.7mm}
\mu z^r, \ \ \iota\hspace{-0.7mm} \circ\hspace{-0.7mm}
g\hspace{-0.7mm} \circ\hspace{-0.7mm} \iota\!^{-\!1}\hspace{-0.7mm}
=\hspace{-0.7mm} z^s
\end{eqnarray*}
where $r\hspace{-0.7mm} =\hspace{-0.7mm} \deg{f}$ and
$s\hspace{-0.7mm} =\hspace{-0.7mm} \deg{g}$.
\end{proposition}

\begin{proof} By the commensurability  assumption, for $m\hspace{-0.7mm}
\in\hspace{-0.7mm}\NN$ there exists $n\hspace{-0.7mm}
\in\hspace{-0.7mm}\NN$ and $\{h_1, h_2\}\hspace{-0.7mm}
\subset\hspace{-0.7mm}\mathrm{End}(\EE)$ such that
\begin{eqnarray*}
f^{n}\hspace{-0.7mm} =\hspace{-0.7mm} g^{m}\hspace{-0.7mm}
\circ\hspace{-0.7mm} h_1, \ \ \ \  g^{n}\hspace{-0.7mm}
=\hspace{-0.7mm} f^{m}\hspace{-0.7mm} \circ\hspace{-0.7mm} h_2.
\end{eqnarray*}
\noindent {\it Case {\it (i)}},  there exist $\{k,
t\}\hspace{-0.7mm} \subset\hspace{-0.7mm}\NN$ such that
$r^k\hspace{-0.7mm} =\hspace{-0.7mm} s^t$. For any $m\hspace{-0.7mm}
\in\hspace{-0.7mm}\NN$ we choose $n_{m}\hspace{-0.7mm}
\in\hspace{-0.7mm}\NN$  and $\varepsilon_m\hspace{-0.7mm}
\in\hspace{-0.7mm} \mathrm{End}(\EE)$ for which
$f^{mk}\hspace{-0.7mm} \circ\hspace{-0.7mm}
\varepsilon_m\hspace{-0.7mm} =\hspace{-0.7mm} g^{n_m}$ or
equivalently $ f^{mk}\hspace{-0.7mm} \circ\hspace{-0.7mm}
\varepsilon_m\hspace{-0.7mm} =\hspace{-0.7mm} g^{mt}\hspace{-0.7mm}
\circ\hspace{-0.7mm} g^{n_m\mathopen -\mathclose mt}. $ The
condition $r^k\hspace{-0.7mm} =\hspace{-0.7mm} s^t$ leads to
$\deg{f^{mk}}\hspace{-0.7mm} =\hspace{-0.7mm} \deg{g^{mt}}$, and
then Proposition \ref{rigidity of Blaschke} implies that there
exists $\iota_m\hspace{-0.7mm} \in\hspace{-0.7mm}\mathrm{Aut}(\EE)$
for which $g^{mt}\hspace{-0.7mm} =\hspace{-0.7mm}
f^{mk}\hspace{-0.7mm} \circ\hspace{-0.7mm} \iota_m$. Consequently
for all positive integer $m$ we have $(f^k)^m\hspace{-0.7mm}
=\hspace{-0.7mm} (f^k\hspace{-0.7mm} \circ\hspace{-0.7mm}
\iota_1)^m\hspace{-0.7mm} \circ\hspace{-0.7mm} \iota_m^{-\!1}$, and
this reduces to Lemma \ref{degf=degg}.

\noindent {\it Case {\it (ii)}}, there exists no
$\iota\hspace{-0.7mm} \in\hspace{-0.7mm}\mathrm{Aut}(\EE)$ for which
$\iota\hspace{-0.7mm} \circ\hspace{-0.7mm} g\hspace{-0.7mm}
\circ\hspace{-0.7mm} \iota^{-\!1}\hspace{-0.7mm} =\hspace{-0.7mm}
z^s$. For any $m\hspace{-0.7mm} \in\hspace{-0.7mm}\NN$ we denote by
$n_{m}$ the minimal integer that $g^{n_m}\hspace{-0.7mm}
=\hspace{-0.7mm} f^{m}\hspace{-0.7mm} \circ\hspace{-0.7mm}
\varepsilon_m$ for some $\varepsilon_m\hspace{-0.7mm}
\in\hspace{-0.7mm}\mathrm{End}(\EE)$. The minimality of $n_m$ forces
that there exists no $t\hspace{-0.7mm}
\subset\hspace{-0.7mm}\mathrm{End}(\EE)$ for which
$\varepsilon_m\hspace{-0.7mm} =\hspace{-0.7mm} t\hspace{-0.7mm}
\circ\hspace{-0.7mm} g$, and therefore by Corollary \ref{corollary
of rigidity} there exist positive integers $m\hspace{-0.7mm}
<\hspace{-0.7mm} p$ such that $\deg{\varepsilon_m}\hspace{-0.7mm}
=\hspace{-0.7mm} \deg{\varepsilon_p}$. This leads to
$s^{n_p\mathopen -\mathclose n_m}\hspace{-0.7mm} =\hspace{-0.7mm}
r^{p\mathopen -\mathclose m}$ which reduces the problem to the
previous case.

\noindent {\it Case {\it (iii)}}, there exists $\iota\hspace{-0.7mm}
\in\hspace{-0.7mm}\mathrm{Aut}(\EE)$ such that $\iota\hspace{-0.7mm}
\circ\hspace{-0.7mm} g\hspace{-0.7mm} \circ\hspace{-0.7mm}
\iota^{-1}\hspace{-0.7mm} =\hspace{-0.7mm} z^s.$ For
$m\hspace{-0.7mm} \in\hspace{-0.7mm}\NN$ there exist
$\varepsilon_m\hspace{-0.7mm} \in\hspace{-0.7mm}\mathrm{End}(\EE)$
and $n_m\hspace{-0.7mm} \in\hspace{-0.7mm}\NN$ such that
$f^{m}\hspace{-0.7mm} \circ\hspace{-0.7mm}
\varepsilon_m\hspace{-0.7mm} =\hspace{-0.7mm} g^{n_m}$ or
equivalently
\begin{eqnarray*}
f^{m}\hspace{-0.7mm} \circ\hspace{-0.7mm}
\varepsilon_m\hspace{-0.7mm} \circ\hspace{-0.7mm}
\iota\!^{-1}\hspace{-0.7mm} =\hspace{-0.7mm}
\iota\!^{-1}\hspace{-0.7mm} \circ\hspace{-0.7mm}
z^{r^m}\hspace{-0.7mm} \circ\hspace{-0.7mm} z^{s^{n_m}/r^m}.
\end{eqnarray*}
Proposition \ref{rigidity of Blaschke} implies that
$f^{m}\hspace{-0.7mm} \sim \hspace{-0.7mm} z^{r^m}$, and then Lemma
\ref{special ramified} applies.
\end{proof}

We then prove
\begin{proposition}\label{noncommensurable}
If $\{f, g\} \hspace{-0.7mm} \subset\hspace{-0.7mm}
\mathrm{End}(\EE)\mathopen \setminus\mathclose \mathrm{Aut}(\EE)$
are non-commensurable and if for all $\{m,n\}\hspace{-0.7mm}
\subset\hspace{-0.7mm}\NN$ the curve $ \PP^1\mathopen
\times_{f^n,\,g^m}\mathclose \PP^1$ admits a Faltings factor then
there exist $\iota\hspace{-0.7mm} \in\hspace{-0.7mm}
\mathrm{Aut}(\EE)$ and $\mu\hspace{-0.7mm} \in\hspace{-0.7mm} \TT$
such that
\begin{eqnarray*}
\iota\hspace{-0.7mm} \circ\hspace{-0.7mm} f\hspace{-0.7mm}
\circ\hspace{-0.7mm} \iota^{-1}\hspace{-0.7mm} =\hspace{-0.7mm} z^r,
\ \ \ \iota\hspace{-0.7mm} \circ\hspace{-0.7mm} g\hspace{-0.7mm}
\circ\hspace{-0.7mm} \iota^{-1}\hspace{-0.7mm} =\hspace{-0.7mm} \mu
z^s
\end{eqnarray*}
where $r\hspace{-0.7mm} =\hspace{-0.7mm} \deg{f}$ and
$s\hspace{-0.7mm} =\hspace{-0.7mm} \deg{g}$.
\end{proposition}
\begin{proof}
By the assumption that $f$ and $g$ are non-commensurable there
exists $t\hspace{-0.7mm} \in\hspace{-0.7mm}\NN$ such that for any
$n\hspace{-0.7mm} \in\hspace{-0.7mm}\NN$ and for any
$h\hspace{-0.7mm} \in\hspace{-0.7mm}\mathrm{End}(\EE)$
\begin{eqnarray}\label{neq}
g^{n}\hspace{-0.7mm} \neq\hspace{-0.7mm} f^{t}\hspace{-0.7mm}
\circ\hspace{-0.7mm} h.
\end{eqnarray}
Given $\{i,j\}\hspace{-0.7mm} \subset\hspace{-0.7mm}\NN$, by the
existence of Faltings factor of $\PP^1\mathopen \times_{(f^{t})^i,\,
g^{j}}\mathclose \PP^1$ and by Theorem \ref{Faltings factor}  there
exist $\{a_{ij}, b_{ij}, c_{ij}\} \hspace{-0.7mm}
\subset\hspace{-0.7mm} \mathrm{End}(\EE)$ and $\{\epsilon_{ij},
\varepsilon_{ij}\}\hspace{-0.7mm}
\subset\hspace{-0.7mm}\mathrm{Aut}(\EE)$ such that
\begin{eqnarray*}
(f^{t})^i\hspace{-0.7mm} =\hspace{-0.7mm} a_{ij}\hspace{-0.7mm}
\circ\hspace{-0.7mm} b_{ij}\hspace{-0.7mm} \circ\hspace{-0.7mm}
\epsilon_{ij}, \ \ \ g^{j}\hspace{-0.7mm} =\hspace{-0.7mm}
a_{ij}\hspace{-0.7mm} \circ\hspace{-0.7mm} c_{ij}\hspace{-0.7mm}
\circ\hspace{-0.7mm} \varepsilon_{ij},
\end{eqnarray*}
where the set $\{b_{ij}, c_{ij}\}$ is described in Theorem
\ref{Faltings factor}. We write $\mathcal{S}\hspace{-0.7mm}
=\hspace{-0.7mm} \{ \deg{a_{ij}}: (i, j) \hspace{-0.7mm}
\in\hspace{-0.7mm} \NN \mathopen \times \mathclose \NN \}$ and
consider the following two cases.

\noindent {\it Case {\it (i)}}, the cardinality of $\calS$ is
infinite.

Given any $h\hspace{-0.7mm} \in\hspace{-0.7mm}\mathrm{End}(\EE)$ and
any pair $\{i, j\}\hspace{-0.7mm} \subset\hspace{-0.7mm}\NN$ we have
\begin{eqnarray}\label{reason}
a_{ij}\hspace{-0.7mm} \neq\hspace{-0.7mm} f^{t}\hspace{-0.7mm}
\circ\hspace{-0.7mm} h.
\end{eqnarray}
Otherwise $g^{j}\hspace{-0.7mm} =\hspace{-0.7mm}
a_{ij}\hspace{-0.7mm} \circ\hspace{-0.7mm} c_{ij}\hspace{-0.7mm}
\circ\hspace{-0.7mm} \varepsilon_{ij}\hspace{-0.7mm}
=\hspace{-0.7mm} f^{t}\hspace{-0.7mm} \circ\hspace{-0.7mm}
(h\hspace{-0.7mm} \circ\hspace{-0.7mm} c_{ij}\hspace{-0.7mm}
\circ\hspace{-0.7mm} \varepsilon_{ij})$, a contradiction to
(\ref{neq}). Because the cardinality of $\calS$ is infinite,
Corollary \ref{corollary of rigidity} and (\ref{reason}) applied to
$a_{ij}\hspace{-0.7mm} \circ\hspace{-0.7mm} (b_{ij}\hspace{-0.7mm}
\circ\hspace{-0.7mm} \epsilon_{ij}) \hspace{-0.7mm} =\hspace{-0.7mm}
(f^{t})^{i} $ shows that $\iota\hspace{-0.7mm} \circ
f^t\hspace{-0.7mm} \circ\hspace{-0.7mm} \iota^{-1}\hspace{-0.7mm}
=\hspace{-0.7mm} z^{rt}$ for some $\iota\hspace{-0.7mm}
\in\hspace{-0.7mm}\mathrm{Aut}(\EE)$. In particular $f^t$ is totally
ramified, and this together with Lemma \ref{special ramified} leads
to the existence of $\sigma\hspace{-0.7mm} \in\hspace{-0.7mm}
\mathrm{Aut}(\EE)$ for which $\sigma\hspace{-0.7mm}
\circ\hspace{-0.7mm} f\hspace{-0.7mm} \circ\hspace{-0.7mm}
\sigma^{-1}\hspace{-0.7mm} =\hspace{-0.7mm} z^{r}$. Neither the
hypothesis nor the conclusion are affected under
\begin{eqnarray*}
f \mapsto \sigma\hspace{-0.7mm} \circ\hspace{-0.7mm}
f\hspace{-0.7mm} \circ\hspace{-0.7mm} \sigma^{-1}, \ \ \ g \mapsto
\sigma\hspace{-0.7mm} \circ\hspace{-0.7mm} g\hspace{-0.7mm}
\circ\hspace{-0.7mm} \sigma^{-1}.
\end{eqnarray*}
We may assume $f(z)\hspace{-0.7mm} =\hspace{-0.7mm} z^r$ and then
$b_{ij}$, as a factor of $z^{rti}$, is totally ramified. Henceforth
we always assume that $i$ is so large compared to $j$ that
$\deg{b_{ij}}\hspace{-0.7mm} >\hspace{-0.7mm} \deg{c_{ij}}$, which
forces $\{ b_{ij}, c_{ij} \}$ falling into case{\it(i)} of Theorem
\ref{Faltings factor} and $c_{ij}$ being totally ramified. By the
remark after Theorem \ref{Faltings factor} we assume $\{b_{ij},
c_{ij}\}\hspace{-0.7mm} =\hspace{-0.7mm} \{z^{\hat{m}},
z^{\hat{r}}\}$, namely
\begin{eqnarray}\label{abcd}
(z^{tr})^i\hspace{-0.7mm} =\hspace{-0.7mm} a_{ij}\hspace{-0.7mm}
\circ\hspace{-0.7mm} b_{ij}\hspace{-0.7mm} \circ\hspace{-0.7mm}
\epsilon_{ij}, \ \ \ g^j\hspace{-0.7mm} =\hspace{-0.7mm}
a_{ij}\hspace{-0.7mm} \circ\hspace{-0.7mm} c_{ij}\hspace{-0.7mm}
\circ\hspace{-0.7mm} \varepsilon_{ij}
\end{eqnarray}
with $b_{ij}\hspace{-0.7mm} =\hspace{-0.7mm} z^{\hat{m}},
c_{ij}\hspace{-0.7mm} =\hspace{-0.7mm} z^{\hat{r}}$ and
$\{\epsilon_{ij}, \varepsilon_{ij}\}\hspace{-0.7mm}
\subset\hspace{-0.7mm} \mathrm{Aut}(\EE).$ It is clear that for
nonlinear $a$ and $b$, $a\hspace{-0.7mm} \circ\hspace{-0.7mm} b$ is
totally ramified if and only if  $a$ and $b$ are and
$|\mathfrak{D}_{a}|\hspace{-0.7mm} =\hspace{-0.7mm}
\mathfrak{d}_{b}$. Our factorization of $(z^{tr})^i$ implies that
$a_{ij}$ is totally ramified  and
$|\mathfrak{D}_{a_{ij}}|\hspace{-0.7mm} =\hspace{-0.7mm}
\mathfrak{d}_{b_{ij}}$. This together with $b_{ij}\hspace{-0.7mm}
=\hspace{-0.7mm} z^{\hat{m}}$ and $c_{ij}\hspace{-0.7mm}
=\hspace{-0.7mm} z^{\hat{r}}$ leads to
$|\mathfrak{D}_{a_{ij}}|\hspace{-0.7mm} =\hspace{-0.7mm}
\mathfrak{d}_{c_{ij}}$ and therefore $g^j$, which equals
$a_{ij}\hspace{-0.7mm} \circ\hspace{-0.7mm} c_{ij}\hspace{-0.7mm}
\circ\hspace{-0.7mm} \varepsilon_{ij}$, is also totally ramified. By
Corollary \ref{special ramified} applied to $g^j$ there exists
$\overline{\iota}\hspace{-0.7mm}
\in\hspace{-0.7mm}\mathrm{Aut}(\EE)$ such that
$\overline{\iota}\hspace{-0.7mm} \circ\hspace{-0.7mm}
g\hspace{-0.7mm} \circ\hspace{-0.7mm}
\overline{\iota}^{-1}\hspace{-0.7mm} =\hspace{-0.7mm} z^{s}$. It is
clear from (\ref{abcd}) that $\mathfrak{d}_{g^j}\hspace{-0.7mm}
=\hspace{-0.7mm} \mathfrak{d}_{a_{ij}}\hspace{-0.7mm}
=\hspace{-0.7mm} \mathfrak{d}_{z^{tri}} \hspace{-0.7mm}
=\hspace{-0.7mm} 0$ and therefore $\mathfrak{d}_{g}\hspace{-0.7mm}
=\hspace{-0.7mm} 0$. This together with
$\overline{\iota}\hspace{-0.7mm} \circ\hspace{-0.7mm}
g\hspace{-0.7mm} \circ\hspace{-0.7mm}
\overline{\iota}^{-1}\hspace{-0.7mm} =\hspace{-0.7mm} z^{s}$ implies
that $g(z)\hspace{-0.7mm} =\hspace{-0.7mm} \mu z^s$ for some
$\mu\hspace{-0.7mm} \in\hspace{-0.7mm} \TT$ .

\noindent {\it Case {\it (ii)}}, the cardinality of $\calS$ is
finite.

If $(r, s)\hspace{-0.7mm} =\hspace{-0.7mm} 1$ then  we always have
$\deg{a_{ij}}\hspace{-0.7mm} =\hspace{-0.7mm} 1$. We only consider
the case that $i$ and $j$ are at least three. By Corollary
\ref{special cheby} neither $f^{ti}$ nor $g^{j}$ is elliptic, and
therefore $\{b_{ij}, c_{ij}\}$ falls into case{\it (i)} of Theorem
\ref{Faltings factor}. The one of $\{b_{ij}, c_{ij}\}$ with smaller
degree must be totally ramified, and so is $\{f^{ti}, g^j\}$ as
$\deg{a_{ij}}\hspace{-0.7mm} =\hspace{-0.7mm} 1$.  Choose either $i$
or $j$ to be arbitrary large, we deduce that $f^{3t}$ and $g^3$ are
both totally ramified.  By Lemma \ref{special ramified} there exist
$\{\epsilon, \varepsilon\}\hspace{-0.7mm}
\subset\hspace{-0.7mm}\mathrm{Aut}(\EE)$ such that
$\epsilon\hspace{-0.7mm} \circ\hspace{-0.7mm} f\hspace{-0.7mm}
\circ\hspace{-0.7mm} \epsilon^{-1}\hspace{-0.7mm} =\hspace{-0.7mm}
z^r$ and $\varepsilon\hspace{-0.7mm} \circ\hspace{-0.7mm}
g\hspace{-0.7mm} \circ\hspace{-0.7mm}
\varepsilon^{-1}\hspace{-0.7mm} =\hspace{-0.7mm} z^s$. It remains to
show $\mathfrak{d}_{f}\hspace{-0.7mm} =\hspace{-0.7mm}
\mathfrak{d}_{g}$. Indeed by remark made at the end of section
\ref{sectionbilutichy} we have $\mathfrak{d}_{b_{ij}}
\hspace{-0.7mm} =\hspace{-0.7mm}
\mathfrak{d}_{c_{ij}}\hspace{-0.7mm} =\hspace{-0.7mm}0$ and
therefore $\mathfrak{d}_{f} \hspace{-0.7mm} =\hspace{-0.7mm}
\mathfrak{d}_{g} \hspace{-0.7mm} =\hspace{-0.7mm} a_{ij}(0)$.

If $(r, s)\hspace{-0.7mm} \neq\hspace{-0.7mm} 1$ then by the
finiteness of $|\calS|$ we have $\min{\{\deg{b_{ij}},
\deg{c_{ij}}\}} \rightarrow \infty $ as  $\min{\{i, j\}} \rightarrow
\infty$, and hence for large $i$ and $j$ the pair $\{f^{ti},
g^{j}\}$ falls into the case{\it (iv)} of Theorem \ref{Faltings
factor}. If we choose sufficiently large positive integers $p$, $n$
and $m$ with $n\mathopen +\mathclose 3\hspace{-0.7mm}
\leq\hspace{-0.7mm} m$ such that $\deg{a_{p,n}}\hspace{-0.7mm}
=\hspace{-0.7mm} \deg{a_{p, m}}$ then
\begin{eqnarray*}
g^{n}\hspace{-0.7mm} =\hspace{-0.7mm} a_{p,\, n}\hspace{-0.7mm}
\circ\hspace{-0.7mm} c_{p,\, n}\hspace{-0.7mm} \circ\hspace{-0.7mm}
\varepsilon_{p,\, n}, \ \ \ g^{m}\hspace{-0.7mm} =\hspace{-0.7mm}
a_{p,\, m}\hspace{-0.7mm} \circ\hspace{-0.7mm} c_{p,\,
m}\hspace{-0.7mm} \circ\hspace{-0.7mm} \varepsilon_{p,\, m}
\end{eqnarray*}
where $c_{p,\, n}, c_{p,\, m}$ are elliptic. This gives
\begin{eqnarray*}
a_{p,\, n}\hspace{-0.7mm} \circ\hspace{-0.7mm} \big(c_{p,\,
n}\hspace{-0.7mm} \circ\hspace{-0.7mm} \varepsilon_{p,\,
n}\hspace{-0.7mm} \circ\hspace{-0.7mm} g^{m\mathopen-\mathclose
n}\big)\hspace{-0.7mm} =\hspace{-0.7mm} a_{p,\, m}\hspace{-0.7mm}
\circ\hspace{-0.7mm} \big(c_{p,\, m}\hspace{-0.7mm}
\circ\hspace{-0.7mm} \varepsilon_{p,\, m}\big).
\end{eqnarray*}
Lemma \ref{simple rigidity} gives $c_{p,\, n}\hspace{-0.7mm}
\circ\hspace{-0.7mm} \varepsilon_{p,\, n}\hspace{-0.7mm}
\circ\hspace{-0.7mm} g^{m\mathopen-\mathclose
n}\hspace{-0.7mm}\sim\hspace{-0.7mm}c_{p,\, m}$ and then Corollary
\ref{elliptic factor elliptic} implies that $g^{m-n}$ is elliptic, a
contradiction to Corollary \ref{special cheby}.
\end{proof}

Combining previous results, we readily prove the main theorem.

\begin{proofof}{Proof of Theorem \ref{disk}}
The infiniteness of $\calO_{f}(x)\hspace{-0.7mm} \cap\hspace{-0.7mm}
\calO_{g}(y)$ implies that for all positive integers $m, n$ there
are infinitely many rational points on $\PP^1\mathopen
\times_{f^{n}, \, g^{m}}\mathclose \PP^1$ over the absolute field
$k$ generated by all coefficients of $f, g$ and $x$, $y$. Indeed by
assumption for $i\hspace{-0.7mm} \ge\hspace{-0.7mm} 1$ there exist
pairwise distinct points $p_i\hspace{-0.7mm} \in
\hspace{-0.7mm}\PP^1$ and $\{n_i,
m_i\}\hspace{-0.7mm}\subset\hspace{-0.7mm}\NN$ such that
\begin{eqnarray*}
f^{n_i}(x) \hspace{-0.7mm} = \hspace{-0.7mm}p_i, \ \ \ g^{m_i}(y) =
\hspace{-0.7mm} p_i.
\end{eqnarray*}
It is clear that $n_i$\,(resp.\,$m_i$) are also pairwise distinct
and therefore tends to infinity as $i$ goes to infinity. Therefore
for $i$ sufficiently large, $(f^{n_i\mathopen -\mathclose n}(x),
g^{m_i\mathopen -\mathclose m}(y))$ are $k$-rational points of
$\PP^1\mathopen \times_{f^{n}, \, g^{m}}\mathclose \PP^1$. These
points are pairwise distinct, as otherwise $x$ would be preperiodic
for $f$ which contradicts to the infiniteness assumption.

By Faltings' theorem the curve $\PP^1\mathopen \times_{f^{n}, \,
g^{m}}\mathclose \PP^1$ has a Faltings factor. If $f$ and $g$ have
no common iteration then by Proposition \ref{commensurable} and
Proposition \ref{noncommensurable} we may assume that
$f\hspace{-0.7mm} =\hspace{-0.7mm}  z^r, g\hspace{-0.7mm}
=\hspace{-0.7mm}  \mu z^s$ with $\mu\hspace{-0.7mm}
\in\hspace{-0.7mm} \TT$. This case is already discussed in
\cite{GTZ10}.
\end{proofof}

It is crucial to require that $f, g\hspace{-0.7mm} \not
\in\hspace{-0.7mm} \mathrm{Aut}(\EE)$ in Theorem \ref{disk}. For an
example we consider $\HH$ instead of $\EE$ and simply take
$f(z)\hspace{-0.7mm} =\hspace{-0.7mm} z\hspace{-0.7mm}
+\hspace{-0.7mm}  1$, $g(z)\hspace{-0.7mm} =\hspace{-0.7mm}  2 z$
and $x\hspace{-0.7mm} =\hspace{-0.7mm} y\hspace{-0.7mm}
=\hspace{-0.7mm}  1$.

\end{document}